\documentclass[reqno,12pt]{amsart}
\usepackage{amscd,amsfonts,amssymb}
\numberwithin{equation}{section}
\newtheorem{theorem}{Theorem}[section]
\newtheorem{lemma}{Lemma}[section]
\newtheorem{corollary}{Corollary}[section]

\theoremstyle{definition}

\theoremstyle{remark}
\newtheorem{remark}{Remark}[section]
\newtheorem{example}{Example}[section]

\newcommand{\essinf}{\mathop{\rm ess \, inf}\limits}
\newcommand{\esssup}{\mathop{\rm ess \, sup}\limits}
\newcommand{\mes}{\mathop{\rm mes}\nolimits}
\newcommand{\diver}{\mathop{\rm div}\nolimits}
\author{Andrej A. Kon'kov}
\address{Department of Differential Equations,
Faculty of Mechanics and Mathematics,
Mo\-s\-cow Lo\-mo\-no\-sov State University,
Vorobyovy Gory,
119992 Moscow, Russia}
\email{konkov@mech.math.msu.su}
\title[]{On properties of solutions of quasilinear second-order elliptic inequalities}
\thanks{The research was supported by RFBR, grant 11-01-12018-ofi-m-2011.}
\date{}

\begin{document}

\begin{abstract}

Let $\Omega$ be an unbounded open subset of ${\mathbb R}^n$, $n \ge 2$,
and $A : \Omega \times {\mathbb R}^n \to {\mathbb R}^n$ 
be a function such that
$$
	C_1
	|\zeta|^p
	\le
	\zeta
	A (x, \zeta),
	\quad
	|A (x, \zeta)|
	\le
	C_2
	|\zeta|^{p-1}
$$
with some constants
$C_1 > 0$,
$C_2 > 0$,
and
$p > 1$
for almost all
$x \in \Omega$
and for all
$\zeta \in {\mathbb R}^n$.
We obtain blow-up conditions and priori estimates for inequalities of the form
$$
	\diver A (x, D u)
	+
	b (x) |D u|^\alpha
	\ge
	q (x) g (u)
	\quad
	\mbox{in } \Omega,
$$
where $p - 1 \le \alpha \le p$ is a real number and, moreover,
$b \in L_{\infty, loc} (\Omega)$,
$q \in L_{\infty, loc} (\Omega)$,
and
$g \in C ([0, \infty))$
are non-negative functions.

\end{abstract}

\maketitle

\section{Introduction}\label{intro}

Let $\Omega$ be an unbounded open subset of ${\mathbb R}^n$, $n \ge 2$,
and $A : \Omega \times {\mathbb R}^n \to {\mathbb R}^n$ 
be a function such that
$$
	C_1
	|\zeta|^p
	\le
	\zeta
	A (x, \zeta),
	\quad
	|A (x, \zeta)|
	\le
	C_2
	|\zeta|^{p-1}
$$
with some constants
$C_1 > 0$,
$C_2 > 0$,
and
$p > 1$
for almost all
$x \in \Omega$
and for all
$\zeta \in {\mathbb R}^n$.

Denote
$
	\Omega_{r_1, r_2}
	=
	\{
		x \in \Omega : r_1 < |x| < r_2
	\}
$
and
$
	B_{r_1, r_2}
	=
	\{
		x \in {\mathbb R}^n : r_1 < |x| < r_2
	\},
$
$0 < r_1 < r_2$.
By $B_r^z$ and $S_r^z$ we mean the open ball and the sphere in ${\mathbb R}^n$ 
of radius $r > 0$ and center at a point $z$.
In the case of $z = 0$, we write $B_r$ and $S_r$
instead of $B_r^0$ and $S_r^0$, respectively.

As in~\cite{LU}, the space ${W_{p, loc}^1 (\Omega)}$ is
the set of measurable functions belonging to ${W_p^1 (B_r \cap \Omega)}$ 
for all real numbers $r > 0$ such that
$B_r \cap \Omega \ne \emptyset$.
The space ${L_{p, loc} (\Omega)}$ is defined analogously.

We consider inequalities of the form
\begin{equation}
	\diver A (x, D u)
	+
	b (x) |D u|^\alpha
	\ge
	q (x) g (u)
	\quad
	\mbox{in } \Omega,
	\label{1.1}
\end{equation}
where 
$D = (\partial / \partial x_1,\ldots,\partial / \partial x_n)$
is the gradient operator,
$p - 1 \le \alpha \le p$ is a real number
and, moreover,
$b \in L_{\infty, loc} (\Omega)$,
$q \in L_{\infty, loc} (\Omega)$,
and
$g \in C ([0, \infty))$
are non-negative functions with $g (t) > 0$ for all $t > 0$.

A non-negative function 
$
	u 
	\in 
	W_{p, loc}^1 (\Omega)
	\cap
	L_{\infty, loc} (\Omega)
$ 
is called a solution of inequality~\eqref{1.1} if 
$A (x, D u) \in L_{p/(p-1), loc} (\Omega)$ 
and
$$
	- \int_\Omega
	A (x, D u)
	D \varphi
	\, dx
	+
	\int_\Omega
	b (x) |D u|^\alpha
	\varphi
	\, dx
	\ge
	\int_{
		\Omega
	}
	q (x) g (u)
	\varphi
	\, dx
$$
for any non-negative function
$
	\varphi 
	\in 
	C_0^\infty (
		\Omega
	).
$
In addition, we say that
\begin{equation}
	\left.
		u
	\right|_{
		\partial \Omega
	}
	=
	0
	\label{1.2}
\end{equation}
if
$
	\psi u
	\in
	{
		\stackrel{\rm \scriptscriptstyle o}{W}\!\!{}_p^1
		(
			\Omega
		)
	}
$
for any
$
	\psi
	\in
	C_0^\infty 
	(
		{\mathbb R}^n
	).
$
Thus, in the case of $\Omega = {\mathbb R}^n$, condition~\eqref{1.2} is valid for all
$
	u
	\in
	W_{p, loc}^1
	(
		{\mathbb R}^n
	).
$

It is obvious that every solution of the equation
$$
	\diver A (x, D u)
	+
	b (x) |D u|^\alpha
	=
	q (x) g (u)
	\quad
	\mbox{in } \Omega
$$
is also a solution of inequality~\eqref{1.1}.
Such equations and inequalities have traditionally attracted the attention of many mathematicians.
They appear in the continuum mechanics, in particular, 
in the theory of non-Newtonian fluids and non-Newtonian filtrations~\cite{Aronsson, Yang}.
Other important examples arise in connection with the equations describing 
electromagnetic fields in spatially dispersive media~\cite{KS}
and the Matukuma and Batt-Faltenbacher-Horst equations 
that appear in the plasma physics~\cite{BFH, Batt}.
In so doing, of special interest is a phenomenon of the absence of non-trivial solutions 
which is known as the blow-up phenomenon.

Our aim is to obtain blow-up conditions and priori estimates
for solutions of problem~\eqref{1.1}, \eqref{1.2}. 
The questions treated below were investigated mainly for 
nonlinearities of the Emden-Fowler type
$g (t) = t^\lambda$~\cite{BVP, F, KL, MP, MPbook, Veron}.
The case of general nonlinearity without lower-order derivatives 
was studied 
in~\cite{FPR2008, GR, Keller, Osserman}.
For inequalities containing lower-order derivatives, 
blow-up conditions were obtained in~\cite{FPR2009}.
However, these results can not be applied to a class of inequalities, 
e.g., to the inequalities discussed in Examples~\ref{e2.1}--\ref{e2.3}.

Also, it should be noted that the authors of~\cite{FPR2008, FPR2009, GR, Keller, Osserman}
use arguments based on the method of barrier functions.
This method involves additional restrictions on the function $A$ 
in the left-hand side of~\eqref{1.1}; 
therefore, it can not be applied to inequalities of the general form~\eqref{1.1}.
For $\alpha = p - 1$, inequalities~\eqref{1.1} were considered in~\cite{MeJMAA, MeNONANA}.
In the present paper, we managed to generalize results of~\cite{MeNONANA} 
to the case of $p - 1 \le \alpha \le p$.

Throughout the paper, it is assumed that $S_r \cap \Omega \ne \emptyset$ for all $r > r_0$,
where $r_0 > 0$ is some real number. 
For every solution of problem~\eqref{1.1}, \eqref{1.2} we denote
$$
	M (r; u)
	=
	\esssup_{
		S_r
		\cap
		\Omega
	}
	u,
	\quad
	r > r_0,
$$
where the restriction of $u$ to
$
	S_r
	\cap
	\Omega
$
is understood in the sense of the trace and
the $\esssup$ on the right-hand side is with respect to
$(n-1)$-dimensional Lebesgue measure on $S_r$.

We also put
$$
	g_\theta (t)
	=
	\inf_{
		(t / \theta, \theta t)
	}
	g,
	\quad
	t > 0,
	\:
	\theta > 1,
$$
$$
	q_\sigma (r)
	=
	\essinf_{
		\Omega_{r / \sigma, \sigma r}
	}
	q,
	\quad
	r > r_0,
	\:
	\sigma > 1
$$
and
\begin{equation}
	f_\sigma (r)
	=
	\frac{
		q_\sigma (r)	
	}{
		1
		+
		r
		\,
		q_\sigma^{
			(\alpha - p + 1) / \alpha
		}
		(r)
		\esssup_{
			\Omega_{r / \sigma, \sigma r}
		}
		b^{
			(p - 1) / \alpha
		}
	},
	\quad
	r > r_0,
	\:
	\sigma > 1.
	\label{1.3}
\end{equation}

\section{Main Results}

\begin{theorem}\label{t2.1}
Let
\begin{equation}
	\int_1^\infty
	(g_\theta (t) t)^{- 1 / p}
	\,
	dt
	<
	\infty,
	\label{t2.1.1}
\end{equation}
\begin{equation}
	\int_1^\infty
	g^{- 1 / \alpha}_\theta (t)
	\,
	dt
	<
	\infty,
	\label{t2.1.2}
\end{equation}
and
\begin{equation}
	\int_{r_0}^\infty
	(r f_\sigma (r))^{1 / (p - 1)}
	\,
	dr
	=
	\infty
	\label{t2.1.3}
\end{equation}
for some real numbers $\theta > 1$ and $\sigma > 1$.
Then any non-negative solution of~\eqref{1.1}, \eqref{1.2}
is equal to zero almost everywhere in $\Omega$.
\end{theorem}

Theorem~\ref{t2.1} is proved in Section~\ref{Proof_of_Theorems}. 
Now, we demonstrate its exactness.

\begin{example}\label{e2.1}
Consider the inequality
\begin{equation}
	\diver (|D u|^{p-2} D u)
	+
	b (x) |D u|^\alpha
	\ge
	q (x)
	u^\lambda
	\quad
	\mbox{in } {\mathbb R^n},
	\label{e2.1.1}
\end{equation}
where
$b \in L_{\infty, loc} ({\mathbb R^n})$ 
and
$q \in L_{\infty, loc} ({\mathbb R^n})$ 
are non-negative functions such that
\begin{equation}
	b (x) \le b_0 |x|^k,
	\;
	b_0 = const > 0,
	\label{e2.1.5}
\end{equation}
for almost all $x$ in a neighborhood of infinity and
\begin{equation}
	q (x) \sim |x|^l
	\quad
	\mbox{as } x \to \infty,
	\label{e2.1.2}
\end{equation}
i.e.
$
	k_1 |x|^l \le q (x) \le k_2 |x|^l
$
with some constants $k_1 > 0$ and $k_2 > 0$ for almost all $x$ in a neighborhood of infinity.

At first, let
\begin{equation}
	\alpha + l (\alpha - p + 1) + k (p - 1) \le 0
	\label{e2.1.6}
\end{equation}
(this condition implies that the second summand in the denominator on the right in~\eqref{1.3} 
is bounded above by a constant for all $r > r_0$).
According to Theorem~\ref{t2.1}, if
\begin{equation}
	\lambda > \alpha
	\quad
	\mbox{and}
	\quad
	l \ge - p,
	\label{e2.1.3}
\end{equation}
then any non-negative solution of~\eqref{e2.1.1}
is equal to zero almost everywhere in ${\mathbb R}^n$. 
On the other hand, if $n \ge p$ and, moreover,
$$
	\lambda > \alpha
	\quad
	\mbox{and}
	\quad
	l < - p,
$$
then
$$
	u (x)
	=
	\max 
	\{ 
		|x|^{
			(l + p) / (p - \lambda - 1)
		},
		1 
	\}
$$
is a solution of~\eqref{e2.1.1}, 
where $b \equiv 0$ and $q \in L_{\infty, loc} ({\mathbb R^n})$ 
is a non-negative function satisfying relation~\eqref{e2.1.2}.
This demonstrates the exactness of the second inequality in~\eqref{e2.1.3}.
The first inequality in~\eqref{e2.1.3} is also exact.
Namely, in the case of $\lambda \le \alpha$, it can be shown that~\eqref{e2.1.1}
has a positive solution for all positive functions
$b \in C ({\mathbb R})$ and $q \in C ({\mathbb R^n})$.

Now, let
\begin{equation}
	\alpha + l (\alpha - p + 1) + k (p - 1) > 0.
	\label{e2.1.7}
\end{equation}
If
\begin{equation}
	\lambda > \alpha
	\quad
	\mbox{and}
	\quad
	l \ge k - \alpha,
	\label{e2.1.4}
\end{equation}
then in accordance with Theorem~\ref{t2.1} any non-negative solution of~\eqref{e2.1.1}
is equal to zero almost everywhere in ${\mathbb R}^n$. 
As we have said, the first inequality in~\eqref{e2.1.4} is exact.
The second one is also exact. Really, in the case that
\begin{equation}
	\lambda > \alpha
	\quad
	\mbox{and}
	\quad
	l < k - \alpha,
	\label{e2.1.9}
\end{equation}
putting
$$
	u (x)
	=
	\max 
	\{ 
		|x|^{
			(l - k + \alpha) / (\alpha - \lambda)
		},
		\gamma 
	\},
$$
where $\gamma > 0$ is large enough, we obtain a solution of~\eqref{e2.1.1}
with a non-negative function 
$b \in L_{\infty, loc} ({\mathbb R^n})$ 
such that
\begin{equation}
	b (x) \sim |x|^k
	\quad
	\mbox{as } x \to \infty
	\label{e2.1.8}
\end{equation}
and a non-negative function
$q \in L_{\infty, loc} ({\mathbb R^n})$ 
satisfying relation~\eqref{e2.1.2}.
\end{example}

\begin{example}\label{e2.2}
Let
$b \in L_{\infty, loc} ({\mathbb R^n})$ 
and
$q \in L_{\infty, loc} ({\mathbb R^n})$ 
be non-negative functions such that~\eqref{e2.1.5} holds and, moreover,
\begin{equation}
	q (x) \sim |x|^l \log^m |x|
	\quad
	\mbox{as } x \to \infty.
	\label{e2.2.1}
\end{equation}

At first, we assume that~\eqref{e2.1.6} is valid. By Theorem~\ref{t2.1}, the condition
$$
	\lambda > \alpha
	\quad
	\mbox{and}
	\quad
	l > - p
$$
guarantees that any non-negative solution of~\eqref{e2.1.1} is equal to zero 
almost everywhere in ${\mathbb R}^n$ for all $m \in {\mathbb R}$.
We are interested in the case of the critical exponent $l = -p$.
In this case,~\eqref{e2.1.6} can obviously be written as
$$
	k \le \alpha - p.
$$
In so doing, if
\begin{equation}
	\lambda > \alpha
	\quad
	\mbox{and}
	\quad
	m \ge 1 - p,
	\label{e2.2.2}
\end{equation}
then in accordance with Theorem~\ref{t2.1} any non-negative solution of~\eqref{e2.1.1} 
is equal to zero almost everywhere in ${\mathbb R}^n$.
As noted in Example~\ref{e2.1}, the first inequality in~\eqref{e2.2.2} is exact. 
At the same time, if $n > p$ and, moreover,
$$
	\lambda > \alpha
	\quad
	\mbox{and}
	\quad
	m < 1 - p,
$$
then
$$
	u (x)
	=
	\max 
	\{
		\log^{
			(m + p - 1) / (p - \lambda - 1)
		}
		|x|,
		\gamma
	\}
$$
is a solution of~\eqref{e2.1.1} for enough large $\gamma > 0$, 
where $b \equiv 0$ and
$q \in L_{\infty, loc} ({\mathbb R^n})$ 
is a non-negative function satisfying relation~\eqref{e2.2.1}.
This demonstrates the exactness of the second inequality in~\eqref{e2.2.2}.

Now, assume that~\eqref{e2.1.7} holds.
According to Theorem~\ref{t2.1}, if
$$
	\lambda > \alpha
	\quad
	\mbox{and}
	\quad
	l > k - \alpha,
$$
then any non-negative solution of~\eqref{e2.1.1} 
is equal to zero almost everywhere in ${\mathbb R}^n$ for all $m \in {\mathbb R}$.
Let us consider the case of the critical exponent $l = k - \alpha$.
In this case relation~\eqref{e2.1.7} takes the form
\begin{equation}
	k > \alpha - p.
	\label{e2.2.5}
\end{equation}
By Theorem~\ref{t2.1}, the condition
\begin{equation}
	\lambda > \alpha
	\quad
	\mbox{and}
	\quad
	m \ge - \alpha
	\label{e2.2.3}
\end{equation}
implies that any non-negative solution of~\eqref{e2.1.1} 
is equal to zero almost everywhere in ${\mathbb R}^n$.
The first inequality in~\eqref{e2.2.3} is exact.
We show the exactness of the second inequality.
Let
\begin{equation}
	\lambda > \alpha
	\quad
	\mbox{and}
	\quad
	m < - \alpha.
	\label{e2.2.4}
\end{equation}
Taking $\gamma > 0$ large enough, one can verify that
$$
	u (x)
	=
	\max
	\{
		\log
		^{
			(m + \alpha) / (\alpha - \lambda)
		}
		|x|,
		\gamma
	\}
$$
is a solution of~\eqref{e2.1.1} for some non-negative functions
$b \in L_{\infty, loc} ({\mathbb R^n})$ 
and
$q \in L_{\infty, loc} ({\mathbb R^n})$ 
satisfying relations~\eqref{e2.1.8} and~\eqref{e2.2.1}, respectively.
\end{example}

\begin{example}\label{e2.3}
Consider the inequality
\begin{equation}
	\diver (|D u|^{p-2} D u)
	+
	b (x) |D u|^\alpha
	\ge
	q (x)
	u^\alpha
	\log^\mu (1 + u)
	\quad
	\mbox{in } {\mathbb R^n},
	\label{e2.3.1}
\end{equation}
where
$b \in L_{\infty, loc} ({\mathbb R^n})$ 
and
$q \in L_{\infty, loc} ({\mathbb R^n})$ 
are non-negative functions such that conditions~\eqref{e2.1.5} and~\eqref{e2.1.2} hold.
We denote $\mu_* = \alpha$ for $\alpha > p - 1$ and $\mu_* = p$ for $\alpha = p - 1$.

Let~\eqref{e2.1.6} be valid. If
\begin{equation}
	\mu > \mu_*
	\quad
	\mbox{and}
	\quad
	l \ge - p,
	\label{e2.3.2}
\end{equation}
then in accordance with Theorem~\ref{t2.1} any non-negative solution of~\eqref{e2.3.1}
is equal to zero almost everywhere in ${\mathbb R}^n$.
The first inequality in~\eqref{e2.3.2} is exact. Namely, if $\mu \le \mu_*$, then~\eqref{e2.3.1}
has a positive solution for all positive functions
$b \in C ({\mathbb R^n})$ 
and
$q \in C ({\mathbb R^n})$.
In the case that $n \ge p$ and, moreover,
$$
	\mu > \mu_*
	\quad
	\mbox{and}
	\quad
	l < - p,
$$
we can also specify a positive solution of~\eqref{e2.3.1}, 
where $b \equiv 0$ and $q \in L_{\infty, loc} ({\mathbb R^n})$ is a non-negative function
satisfying relation~\eqref{e2.1.2}.
This solution is given by
$$
	u (x)
	=
	\max
	\{
		|x|^{
			(l + p) / (p - \alpha - 1)
		}
		\log^{
			\mu / (p - \alpha - 1)
		}
		|x|,
		\gamma
	\}
$$
for $\alpha > p - 1$ and
$$
	u (x)
	=
	e^{
		\max
		\{
			|x|^{
				(l + p) / (p - \mu)
			},
			\gamma
		\}
	}
$$
for $\alpha = p - 1$, where $\gamma > 0$ is large enough.
Hence, the second inequality in~\eqref{e2.3.2} is exact too.

Assume now that~\eqref{e2.1.7} is fulfilled.
By Theorem~\ref{t2.1}, if
\begin{equation}
	\mu > \mu_*
	\quad
	\mbox{and}
	\quad
	l \ge k - \alpha,
	\label{e2.3.3}
\end{equation}
then any non-negative solution of~\eqref{e2.3.1} is equal to zero 
almost everywhere in ${\mathbb R}^n$.
As we have previously said, the first inequality in~\eqref{e2.3.3} is exact.
Let us show the exactness of the second inequality.
Really, in the case that 
$$
	\mu > \mu_*
	\quad
	\mbox{and}
	\quad
	l < k - \alpha,
$$
putting
$$
	u (x)
	=
	e^{
		\max
		\{
			|x|^{
				(l - k + \alpha) / (\alpha - \mu)
			},
			\gamma
		\}
	},
$$
where $\gamma > 0$ is large enough, we obtain a solution of~\eqref{e2.3.1}
with non-negative functions 
$b \in L_{\infty, loc} ({\mathbb R^n})$ 
and
$q \in L_{\infty, loc} ({\mathbb R^n})$ 
satisfying relations~\eqref{e2.1.8} and~\eqref{e2.1.2}, respectively.
\end{example}

\begin{theorem}\label{t2.2}
Let there be real numbers $\theta > 1$ and $\sigma > 1$ such that~\eqref{t2.1.3} is valid
and, moreover, at least one of conditions~\eqref{t2.1.1}, \eqref{t2.1.2} does not hold.
If $u \not\equiv 0$ is a non-negative solution of~\eqref{1.1}, \eqref{1.2}, then
\begin{equation}
	M (r; u)
	\ge
	F_\infty^{- 1}
	\left(
		C
		\int_{r_0}^r
		(\xi f_\sigma (\xi))^{1 / (p - 1)}
		\,
		d\xi
	\right)
	\label{t2.2.1}
\end{equation}
for all sufficiently large $r$, where $F_\infty^{- 1}$ is the function inverse to
$$
	F_\infty (\xi)
	=
	\left(
		\int_1^\xi
		(g_\theta (t) t)^{- 1 / p}
		\,
		dt
	\right)^{p / (p - 1)}
	+
	\int_1^\xi
	g^{- 1 / \alpha}_\theta (t)
	\,
	dt
$$
and the constant $C > 0$ depends only on $n$, $p$, $\theta$, $\sigma$, $\alpha$, $C_1$, and $C_2$.
\end{theorem}

\begin{theorem}\label{t2.3}
Let there be real numbers $\theta > 1$ and $\sigma > 1$ 
such that~\eqref{t2.1.1} and \eqref{t2.1.2} 
are valid and, moreover, condition~\eqref{t2.1.3} does not hold.
Then any non-negative solution of~\eqref{1.1}, \eqref{1.2} satisfies the estimate
\begin{equation}
	M (r; u)
	\le
	F_0^{- 1}
	\left(
		C
		\int_r^\infty
		(\xi f_\sigma (\xi))^{1 / (p - 1)}
		\,
		d\xi
	\right)
	\label{t2.3.1}
\end{equation}
for all sufficiently large $r$, where $F_0^{- 1}$ is the function inverse to
$$
	F_0 (\xi)
	=
	\left(
		\int_\xi^\infty
		(g_\theta (t) t)^{- 1 / p}
		\,
		dt
	\right)^{p / (p - 1)}
	+
	\int_\xi^\infty
	g^{- 1 / \alpha}_\theta (t)
	\,
	dt
$$
and the constant $C > 0$ depends only on $n$, $p$, $\theta$, $\sigma$, $\alpha$, $C_1$, and $C_2$.
\end{theorem}

Theorems~\ref{t2.2} and~\ref{t2.3} are proved in Section~\ref{Proof_of_Theorems}.

\begin{remark}\label{r2.1}
We assume by definition that $F_0 (\infty) = 0$. 
Therefore, $F_0^{- 1} (0) = \infty$ and~\eqref{t2.3.1} is fulfilled automatically if 
the integral in the right-hand side is equal to zero.
\end{remark}

The following examples demonstrate an application of Theorems~\ref{t2.2} and~\ref{t2.3}.

\begin{example}\label{e2.4}
Consider inequality~\eqref{e2.1.1},
where
$b \in L_{\infty, loc} ({\mathbb R^n})$ 
and
$q \in L_{\infty, loc} ({\mathbb R^n})$ 
are non-negative functions such that relations~\eqref{e2.1.5}, \eqref{e2.1.2}, 
and~\eqref{e2.1.7} are valid. 
By Theorem~\ref{t2.2}, if
\begin{equation}
	0 \le \lambda < \alpha
	\quad
	\mbox{and}
	\quad
	l > k - \alpha.
	\label{e2.4.1}
\end{equation}
then any non-negative solution $u \not\equiv 0$ 
of~\eqref{e2.1.1} satisfies the estimate
$$
	M (r; u)
	\ge
	C
	r^{
		(l - k + \alpha) / (\alpha - \lambda)
	}
$$
for all $r$ in a neighborhood of infinity, where the constant $C > 0$ 
does not depend on $u$.

Now, assume that condition~\eqref{e2.1.9} holds instead of~\eqref{e2.4.1}. 
Then in accordance with Theorem~\ref{t2.3} any non-negative solution of~\eqref{e2.1.1} 
satisfies the estimate
$$
	M (r; u)
	\le
	C
	r^{
		(l - k + \alpha) / (\alpha - \lambda)
	}
$$
for all $r$ in a neighborhood of infinity, where the constant $C > 0$ does not depend on $u$.
\end{example}

\begin{example}\label{e2.5}
Let
$b \in L_{\infty, loc} ({\mathbb R^n})$ 
and
$q \in L_{\infty, loc} ({\mathbb R^n})$ 
be non-negative functions such that~\eqref{e2.1.5} and~\eqref{e2.2.5} are valid and, moreover,
$$
	q (x) \sim |x|^{k - \alpha} \log^m |x|
	\quad
	\mbox{as } x \to \infty
$$
In other words, we take the critical exponent $l = k - \alpha$ in formula~\eqref{e2.2.1}.
According to Theorem~\ref{t2.2}, if
\begin{equation}
	0 \le \lambda < \alpha
	\quad
	\mbox{and}
	\quad
	m > - \alpha,
	\label{e2.5.2}
\end{equation}
then any non-negative solution $u \not\equiv 0$ of~\eqref{e2.1.1} satisfies the estimate
$$
	M (r; u)
	\ge
	C
	\log^{
		(m + \alpha) / (\alpha - \lambda)
	}
	r
$$
for all $r$ in a neighborhood of infinity, where the constant $C > 0$ 
does not depend on $u$.

In the case that~\eqref{e2.2.4} holds instead of~\eqref{e2.5.2},
by Theorem~\ref{t2.3}, any non-negative solution of~\eqref{e2.1.1} 
satisfies the estimate
$$
	M (r; u)
	\le
	C
	\log^{
		(m + \alpha) / (\alpha - \lambda)
	}
	r
$$
for all $r$ in a neighborhood of infinity, where the constant $C > 0$ 
does not depend on $u$.
\end{example}

It does not present any particular problem to verify that all estimates 
given in Examples~\ref{e2.4} and~\ref{e2.5} are exact.

\section{Proof of Theorems~\ref{t2.1}--\ref{t2.3}}\label{Proof_of_Theorems}

Throughout this section, we shall assume that $u \not\equiv 0$ 
is a non-negative solution of~\eqref{1.1}, \eqref{1.2}.
We need several preliminary assertions.

\begin{lemma}\label{l3.0}
Let $M (r_1; u) = M (r_2; u) > 0$ for some real numbers $r_0 < r_1 < r_2$. Then
$$
	\essinf_{
		\Omega_{r_1, r_2}
	}
	q
	=
	0.
$$
\end{lemma}

\begin{lemma}\label{l3.1}
Let 
${M (r_2; u)} > {M (r_1; u)} \ge \beta {M (r_2; u)}$
and
$$
	(M (r_2; u) - M (r_1; u))^{\alpha - p + 1}
	(r_2 - r_1)^{p - \alpha}
	\esssup_{
		\Omega_{r_1, r_2}
	}
	b
	\le
	\frac{C_1}{4}
$$
for some real numbers 
$r_0 < r_1 < r_2$ 
and 
$0 < \beta < 1$. Then
$$
	(M (r_2; u) - M (r_1; u))^{p - 1}
	\ge
	C
	(r_2 - r_1)^p
	\essinf_{
		\Omega_{r_1, r_2}
	}
	q
	\inf_{
		(\beta M (r_1; u), M (r_2; u))
	}
	g,
$$
where the constant $C > 0$ depends only on $n$, $p$, $\alpha$, $\beta$, $C_1$, and $C_2$.
\end{lemma}

\begin{lemma}\label{l3.2}
Let 
${M (r_2; u)} > {M (r_1; u)} \ge \beta {M (r_2; u)}$
and
$$
	(M (r_2; u) - M (r_1; u))^{\alpha - p + 1}
	(r_2 - r_1)^{p - \alpha}
	\esssup_{
		\Omega_{r_1, r_2}
	}
	b
	\ge
	\lambda
$$
for some real numbers 
$r_0 < r_1 < r_2$, $0 < \beta < 1$, and $\lambda > 0$.
Then
$$
	(M (r_2; u) - M (r_1; u))^\alpha
	\esssup_{
		\Omega_{r_1, r_2}
	}
	b
	\ge
	C
	(r_2 - r_1)^\alpha
	\essinf_{
		\Omega_{r_1, r_2}
	}
	q
	\inf_{
		(\beta M (r_1; u), M (r_2; u))
	}
	g,
$$
where the constant $C > 0$ depends only on 
$n$, $p$, $\alpha$, $\beta$, $\lambda$, $C_1$, and $C_2$.
\end{lemma}

The proof of Lemmas~\ref{l3.0}--\ref{l3.2} is given in Section~\ref{Proof_of_Lemmas}.

We note that ${M (\cdot; u)}$ is a non-decreasing function on the interval $(r_0, \infty)$
and, moreover, ${M (r - 0; u)} = {M (r; u)}$ for all $r > r_0$ 
(see Corollary~\ref{c4.2}, Section~\ref{Proof_of_Lemmas}).

In the proof of Lemma~\ref{l3.3}, by $c$ we denote various positive 
constants that can depend only on $n$, $p$, $\alpha$, $\eta$, $C_1$, and $C_2$.
In the proof of Lemma~\ref{l3.4}, analogous constants can also depend on $\tau$,
whereas in the proof of Lemma~\ref{l3.5}, they can depend only on
$n$, $p$, $\alpha$, $\theta$, $\sigma$, $C_1$, and $C_2$.

\begin{lemma}\label{l3.3}
Let 
$0 < M (r_1 + 0; u) \le \eta^{- 1 / 2} M (r_2; u)$ 
and 
$r_2 \le \tau r_1$
for some real numbers
$r_0 \le r_1 < r_2$, $\eta > 1$, and $\tau > 1$.
Then at least one of the following two inequalities is valid:
\begin{equation}
	\int_{
		M (r_1 + 0; u)
	}^{
		M (r_2; u)
	}
	(g_\eta (t) t)^{-1 / p}
	\,
	dt
	\ge
	C
	\int_{r_1}^{r_2}
	q_\tau^{1 / p}
	(\xi)
	\,
	d\xi,
	\label{l3.3.1}
\end{equation}
\begin{equation}
	\int_{
		M (r_1 + 0; u)
	}^{
		M (r_2; u)
	}
	g_\eta^{-1 / \alpha}
	(t)
	\,
	dt
	\ge
	C
	\int_{r_1}^{r_2}
	(\xi f_\tau (\xi))^{1 / (p - 1)}
	\,
	d\xi,
	\label{l3.3.2}
\end{equation}
where the constant $C > 0$ depends only on $n$, $p$, $\alpha$, $\eta$, $C_1$, and $C_2$.
\end{lemma}

\begin{proof}
Consider a finite sequence
of real numbers $\rho_k < \ldots < \rho_1 < \rho_0$
defined as follows.
We take $\rho_0 = r_2$. 
Assume that $\rho_i$ is already known. 
In the case of $\rho_i = r_1$, we take $k = i$ and stop; otherwise we put
$$
	\rho_{i+1}
	=
	\inf
	\{
		\xi \in (r_1, \rho_i)
		:
		M (\xi; u) > \eta^{- 1 / 2} M (\rho_i; u)
	\}.
$$
It can easily be seen that this procedure must terminate at a finite step.
In so doing, we have
\begin{equation}
	 \eta^{- 1 / 2} M (\rho_i; u) \le M (\rho_{i + 1} + 0; u),
	 \quad
	 i = 0, \dots, k - 1
	\label{pl3.3.2}
\end{equation}
and
\begin{equation}
	 M (\rho_{i + 1}; u) \le \eta^{- 1 / 2} M (\rho_i; u),
	 \quad
	 i = 0, \dots, k - 2.
	\label{pl3.3.3}
\end{equation}

Let $\Xi_1$ be the set of integers $i \in \{ 0, \ldots, k - 1 \}$ such that
$M (\rho_i; u) > M (\rho_{i + 1} + 0; u)$
and
\begin{equation}
	(M (\rho_i; u) - M (\rho_{i + 1} + 0; u))^{\alpha - p + 1}
	(\rho_i - \rho_{i + 1})^{p - \alpha}
	\esssup_{
		\Omega_{\rho_{i + 1}, \rho_i}
	}
	b
	\le
	\frac{C_1}{4}.
	\label{pl3.3.12}
\end{equation}
Also let $\Xi_2 = \{ 0, \ldots, k - 1 \} \setminus \Xi_1$.
By Lemma~\ref{l3.1},
\begin{align}
	(M (\rho_i; u) - M (\rho_{i + 1} + 0; u))^{p - 1}
	\ge
	{}
	&
	c
	(\rho_i - \rho_{i + 1})^p
	\essinf_{
		\Omega_{\rho_{i + 1}, \rho_i}
	}
	q
	\nonumber
	\\
	&
	\times
	\inf_{
		(\eta^{- 1 / 2} M (\rho_{i + 1} + 0; u), M (\rho_i; u))
	}
	g
	\label{pl3.3.4}
\end{align}
for all $i \in \Xi_1$.
In turn, Lemmas~\ref{l3.0} and~\ref{l3.2} imply the inequality
\begin{align}
	(M (\rho_i; u) - M (\rho_{i + 1} + 0; u))^\alpha
	\esssup_{
		\Omega_{\rho_{i + 1}, \rho_i}
	}
	b
	\ge
	{}
	&
	c
	(\rho_i - \rho_{i + 1})^\alpha
	\essinf_{
		\Omega_{\rho_{i + 1}, \rho_i}
	}
	q
	\nonumber
	\\
	&
	\times
	\inf_{
		(\eta^{- 1 / 2} M (\rho_{i + 1} + 0; u), M (\rho_i; u))
	}
	g
	\label{pl3.3.5}
\end{align}
for all $i \in \Xi_2$.

Let us show that
\begin{equation}
	\int_{
		M (\rho_k + 0; u)
	}^{
		M (\rho_0; u)
	}
	(g_\eta (t) t)^{- 1 / p}
	\,
	dt
	\ge
	c
	\essinf_{
		\Omega_{r_1, r_2}
	}
	q^{1 / p}
	\sum_{
		i 
		\in 
		\Xi_1 
	}
	(\rho_i - \rho_{i + 1}).
	\label{pl3.3.6}
\end{equation}
Really, taking into account~\eqref{pl3.3.2}, we have
\begin{align*}
	\int_{
		\eta^{- 1 / 2} 
		M (\rho_i; u)
	}^{
		M (\rho_i; u)
	}
	(g_\eta (t) t)^{- 1 / p}
	\,
	dt
	\ge
	{}
	&
	c
	(M (\rho_i; u) - M (\rho_{i + 1} + 0; u))^{(p - 1) / p}
	\\
	&
	\times
	\sup_{
		(\eta^{- 1 / 2} M (\rho_{i + 1} + 0; u), M (\rho_i; u))
	}
	g^{- 1 / p}
\end{align*}
for all $i \in \Xi_1$.
Combining this with the inequality
\begin{align}
	&
	(M (\rho_i; u) - M (\rho_{i + 1} + 0; u))^{(p - 1) / p}
	\sup_{
		(\eta^{- 1 / 2} M (\rho_{i + 1} + 0; u), M (\rho_i; u))
	}
	g^{- 1 / p}
	\nonumber
	\\
	&
	{}
	\quad
	\ge
	c
	(\rho_i - \rho_{i + 1})
	\essinf_{
		\Omega_{\rho_{i + 1}, \rho_i}
	}
	q^{1 / p}
	\label{pl3.3.8}
\end{align}
which is a consequence of~\eqref{pl3.3.4}, we immediately obtain
$$
	\int_{
		\eta^{- 1 / 2} 
		M (\rho_i; u)
	}^{
		M (\rho_i; u)
	}
	(g_\eta (t) t)^{- 1 / p}
	\,
	dt
	\ge
	c
	(\rho_i - \rho_{i + 1})
	\essinf_{
		\Omega_{\rho_{i + 1}, \rho_i}
	}
	q^{1 / p}
$$
for all $i \in \Xi_1$.
By~\eqref{pl3.3.3}, for different $i \in {\{ 0, \ldots, k - 2 \}}$, 
the domains of integration for the integrals in the left-hand side of the last estimate
intersect in a set of zero measure.
Thus, summing this estimate over all 
$i \in \Xi_1 \cap \{ 0, \ldots, k - 2 \}$, 
one can conclude that
\begin{equation}
	\int_{
		M (\rho_{k - 1}; u)
	}^{
		M (\rho_0; u)
	}
	(g_\eta (t) t)^{- 1 / p}
	\,
	dt
	\ge
	c
	\essinf_{
		\Omega_{r_1, r_2}
	}
	q^{1 / p}
	\sum_{
		i 
		\in 
		\Xi_1 
		\cap
		\{ 0, \ldots, k - 2 \}
	}
	(\rho_i - \rho_{i + 1}).
	\label{pl3.3.9}
\end{equation}
If $k - 1 \not\in \Xi_1$, then~\eqref{pl3.3.9} obviously implies~\eqref{pl3.3.6}.
Consider the case of $k - 1 \in \Xi_1$. From~\eqref{pl3.3.2}, it follows that
\begin{align*}
	\int_{
		M (\rho_k + 0; u)
	}^{
		\eta^{1 / 2} 
		M (\rho_k + 0; u)
	}
	(g_\eta (t) t)^{- 1 / p}
	\,
	dt
	\ge
	{}
	&
	c
	(M (\rho_{k - 1}; u) - M (\rho_k + 0; u))^{(p - 1) / p}
	\\
	&
	\times
	\sup_{
		(\eta^{- 1 / 2} M (\rho_k  + 0; u), M (\rho_{k - 1}; u))
	}
	g^{- 1 / p}.
\end{align*}
Combining this with inequality~\eqref{pl3.3.8}, where $i = k - 1$, we have
$$
	\int_{
		M (\rho_k + 0; u)
	}^{
		\eta^{1 / 2} 
		M (\rho_k + 0; u)
	}
	(g_\eta (t) t)^{- 1 / p}
	\,
	dt
	\ge
	c
	(\rho_{k - 1} - \rho_k)
	\essinf_{
		\Omega_{\rho_k, \rho_{k - 1}}
	}
	q^{1 / p}.
$$
Finally, summing the last expression and~\eqref{pl3.3.9}, we obtain~\eqref{pl3.3.6}.

Further, let us show that
\begin{equation}
	\int_{
		M (\rho_k + 0; u)
	}^{
		M (\rho_0; u)
	}
	g_\eta^{- 1 / \alpha} (t)
	\,
	dt
	\ge
	c
	\sup_{
		r \in (r_1, r_2)
	}
	(r f_\tau (r))^{1 / (p - 1)}
	\sum_{
		i 
		\in 
		\Xi_2 
	}
	(\rho_i - \rho_{i + 1}).
	\label{pl3.3.10}
\end{equation}
Taking into account~\eqref{pl3.3.5}, we have
\begin{align}
	&
	(M (\rho_i; u) - M (\rho_{i + 1} + 0; u))
	\sup_{
		(\eta^{- 1 / 2} M (\rho_{i + 1} + 0; u), M (\rho_i; u))
	}
	g^{- 1 / \alpha}
	\nonumber
	\\
	&
	{}
	\quad
	\ge
	c
	(\rho_i - \rho_{i + 1})
	\sup_{
		r \in (r_1, r_2)
	}
	(r f_\tau (r))^{1 / (p - 1)}
	\label{pl3.3.7}
\end{align}
for all $i \in \Xi_2$.
Really, if
$$
	\esssup_{
		\Omega_{\rho_{i + 1}, \rho_i}
	}
	b
	=
	0,
$$
then 
$$
	\essinf_{
		\Omega_{\rho_{i + 1}, \rho_i}
	}
	q
	=
	0
$$
by~\eqref{pl3.3.5}. 
Hence, $f_\tau (r) = 0$ for all $r \in (\rho_{i + 1}, \rho_i)$ and~\eqref{pl3.3.7} is evident.
Now, let
$$
	\esssup_{
		\Omega_{\rho_{i + 1}, \rho_i}
	}
	b
	>
	0,
$$
then~\eqref{pl3.3.5} implies the estimate
\begin{align*}
	&
	(M (\rho_i; u) - M (\rho_{i + 1} + 0; u))
	\sup_{
		(\eta^{- 1 / 2} M (\rho_{i + 1} + 0; u), M (\rho_i; u))
	}
	g^{- 1 / \alpha}
	\nonumber
	\\
	&
	{}
	\quad
	\ge
	c
	(\rho_i - \rho_{i + 1})
	\left(
		\frac{
			\essinf_{
				\Omega_{\rho_{i + 1}, \rho_i}
			}
			q
		}{
			\esssup_{
				\Omega_{\rho_{i + 1}, \rho_i}
			}
			b
		}
	\right)^{ 1 / \alpha}
\end{align*}
and, to establish the validity of~\eqref{pl3.3.7}, it remains to note that
$$
	\frac{
		\essinf_{
			\Omega_{\rho_{i + 1}, \rho_i}
		}
		q
	}{
		\esssup_{
			\Omega_{\rho_{i + 1}, \rho_i}
		}
		b
	}
	\ge
	\sup_{
		r \in (r_1, r_2)
	}
	(r f_\tau (r))^{\alpha / (p - 1)}.
$$
In turn, combining~\eqref{pl3.3.7} with the inequality
\begin{align*}
	\int_{
		M (\rho_{i + 1} + 0; u)
	}^{
		M (\rho_i; u)
	}
	g_\eta^{- 1 / \alpha} (t)
	\,
	dt
	\ge
	{}
	&
	(M (\rho_i; u) - M (\rho_{i + 1} + 0; u))
	\\
	&
	\times
	\sup_{
		(\eta^{- 1 / 2} M (\rho_{i + 1} + 0; u), M (\rho_i; u))
	}
	g^{- 1 / \alpha}
\end{align*}
which follows from~\eqref{pl3.3.2}, we have
\begin{equation}
	\int_{
		M (\rho_{i + 1} + 0; u)
	}^{
		M (\rho_i; u)
	}
	g_\eta^{- 1 / \alpha} (t)
	\,
	dt
	\ge
	c
	(\rho_i - \rho_{i + 1})
	\sup_{
		r \in (r_1, r_2)
	}
	(r f_\tau (r))^{1 / (p - 1)}
	\label{pl3.3.13}
\end{equation}
for all $i \in \Xi_2$. 
The last inequality obviously implies~\eqref{pl3.3.10}.

In the case of
\begin{equation}
	\sum_{
		i 
		\in 
		\Xi_1
	}
	(\rho_i - \rho_{i + 1})
	\ge
	\frac{r_2 - r_1}{2},
	\label{pl3.3.11}
\end{equation}
estimate~\eqref{pl3.3.6} allows us to establish the validity of~\eqref{l3.3.1}.
On the other hand, if~\eqref{pl3.3.11} is not valid, then 
$$
	\sum_{
		i 
		\in 
		\Xi_2
	}
	(\rho_i - \rho_{i + 1})
	\ge
	\frac{r_2 - r_1}{2}
$$
and, using~\eqref{pl3.3.10}, we immediately obtain~\eqref{l3.3.2}.

The proof is completed.
\end{proof}

\begin{lemma}\label{l3.4}
Let $0 < M (r_2; u) \le \eta^{1 / 2} M (r_1 + 0; u)$ 
and 
$\tau^{1 / 2} r_1 \le r_2$
for some real numbers
$r_0 \le r_1 < r_2$, $\eta > 1$, and $\tau > 1$.
Then either estimate~\eqref{l3.3.2} holds or
\begin{equation}
	\int_{
		M (r_1 + 0; u)
	}^{
		M (r_2; u)
	}
	g_\eta^{-1 / (p - 1)}
	(t)
	\,
	dt
	\ge
	C
	\int_{r_1}^{r_2}
	(\xi f_\tau (\xi))^{1 / (p - 1)}
	\,
	d\xi,
	\label{l3.4.1}
\end{equation}
where the constant $C > 0$ depends only on $n$, $p$, $\alpha$, $\eta$, $\tau$, $C_1$, and $C_2$.
\end{lemma}

\begin{proof}
We take the maximal integer $k$ such that $\tau^{k / 2} r_1 \le r_2$. 
By definition, put $\rho_i = \tau^{- i / 2} r_2$, $i = 0, \ldots, k - 1$, and $\rho_k = r_1$.
We obviously have 
$$
	\tau^{1 / 2} 
	\rho_{i + 1} 
	\le 
	\rho_i 
	\le 
	\tau 
	\rho_{i + 1},
	\quad
	i = 0, \ldots, k - 1.
$$
As in the proof of Lemma~\ref{l3.3}, let $\Xi_1$ be the set of integers 
$i \in {\{ 0, \ldots, k - 1 \}}$ such that 
$M (\rho_i; u) > M (\rho_{i + 1} + 0; u)$ and, moreover, 
condition~\eqref{pl3.3.12} is fulfilled.
Also denote $\Xi_2 = {\{ 0, \ldots, k - 1 \}} \setminus \Xi_1$.

From Lemma~\ref{l3.1}, it follows that
\begin{align*}
	&
	(M (\rho_i; u) - M (\rho_{i + 1} + 0; u))
	\sup_{
		(\eta^{- 1 / 2} M (\rho_{i + 1} + 0; u), M (\rho_i; u))
	}
	g^{- 1 / (p - 1)}
	\nonumber
	\\
	&
	{}
	\quad
	\ge
	c
	(\rho_i - \rho_{i + 1})^{p / (p - 1)}
	\essinf_{
		\Omega_{\rho_{i + 1}, \rho_i}
	}
	q^{1 / (p - 1)}
\end{align*}
for all $i \in \Xi_1$. 
Combining this with the evident inequalities
\begin{align*}
	\int_{
		M (\rho_{i + 1} + 0; u)
	}^{
		M (\rho_i; u)
	}
	g_\eta^{- 1 / (p - 1)} (t)
	\,
	dt
	\ge
	{}
	&
	(M (\rho_i; u) - M (\rho_{i + 1} + 0; u))
	\\
	&
	\times
	\sup_{
		(\eta^{- 1 / 2} M (\rho_{i + 1} + 0; u), M (\rho_i; u))
	}
	g^{- 1 / (p - 1)}
\end{align*}
and
$$
	(\rho_i - \rho_{i + 1})^{p / (p - 1)}
	\essinf_{
		\Omega_{\rho_{i + 1}, \rho_i}
	}
	q^{1 / (p - 1)}
	\ge
	c
	\int_{
		\rho_{i + 1}
	}^{
		\rho_i
	}
	(\xi f_\tau (\xi))^{1 / (p - 1)}
	\,
	d\xi,
$$
we obtain
$$
	\int_{
		M (\rho_{i + 1} + 0; u)
	}^{
		M (\rho_i; u)
	}
	g_\eta^{- 1 / (p - 1)} (t)
	\,
	dt
	\ge
	c
	\int_{
		\rho_{i + 1}
	}^{
		\rho_i
	}
	(\xi f_\tau (\xi))^{1 / (p - 1)}
	\,
	d\xi.
$$
Summing the last inequality over all $i \in \Xi_1$, one can conclude that
\begin{equation}
	\int_{
		M (\rho_{k} + 0; u)
	}^{
		M (\rho_0; u)
	}
	g_\eta^{- 1 / (p - 1)} (t)
	\,
	dt
	\ge
	c
	\sum_{i \in \Xi_1}
	\int_{
		\rho_{i + 1}
	}^{
		\rho_i
	}
	(\xi f_\tau (\xi))^{1 / (p - 1)}
	\,
	d\xi.
	\label{pl3.4.1}
\end{equation}
If
\begin{equation}
	\sum_{i \in \Xi_1}
	\int_{
		\rho_{i + 1}
	}^{
		\rho_i
	}
	(\xi f_\tau (\xi))^{1 / (p - 1)}
	\,
	d\xi
	\ge
	\frac{1}{2}
	\int_{
		r_1
	}^{
		r_2
	}
	(\xi f_\tau (\xi))^{1 / (p - 1)}
	\,
	d\xi
	\label{pl3.4.2}
\end{equation}
then~\eqref{pl3.4.1} immediately implies~\eqref{l3.4.1}.
Assume that~\eqref{pl3.4.2} is not valid, then
\begin{equation}
	\sum_{i \in \Xi_2}
	\int_{
		\rho_{i + 1}
	}^{
		\rho_i
	}
	(\xi f_\tau (\xi))^{1 / (p - 1)}
	\,
	d\xi
	\ge
	\frac{1}{2}
	\int_{
		r_1
	}^{
		r_2
	}
	(\xi f_\tau (\xi))^{1 / (p - 1)}
	\,
	d\xi.
	\label{pl3.4.3}
\end{equation}
Repeating the arguments given in the proof of inequality~\eqref{pl3.3.13}
with $r_1$ and $r_2$ replaced by $\rho_{i + 1}$ and $\rho_i$, respectively, 
we have
\begin{align*}
	\int_{
		M (\rho_{i + 1} + 0; u)
	}^{
		M (\rho_i; u)
	}
	g_\eta^{- 1 / \alpha} (t)
	\,
	dt
	&
	{}
	\ge
	c
	(\rho_i - \rho_{i + 1})
	\sup_{
		r \in (\rho_{i + 1}, \rho_i)
	}
	(r f_\tau (r))^{1 / (p - 1)}
	\\
	&
	\ge
	c
	\int_{
		\rho_{i + 1}
	}^{
		\rho_i
	}
	(\xi f_\tau (\xi))^{1 / (p - 1)}
	\,
	d\xi
\end{align*}
for all $i \in \Xi_2$, whence it follows that
$$
	\int_{
		M (\rho_{k} + 0; u)
	}^{
		M (\rho_0; u)
	}
	g_\eta^{- 1 / \alpha} (t)
	\,
	dt
	\ge
	c
	\sum_{i \in \Xi_2}
	\int_{
		\rho_{i + 1}
	}^{
		\rho_i
	}
	(\xi f_\tau (\xi))^{1 / (p - 1)}
	\,
	d\xi.
$$
Combining this with~\eqref{pl3.4.3} we obtain estimate~\eqref{l3.3.2}.

The proof is completed.
\end{proof}

\begin{lemma}\label{l3.5}
Let $M_1 \le M (r_1 + 0; u) \le M (r_2; u) \le M_2$,
$\sigma r_1 \le r_2$, and $\theta M_1 \le M_2$,
where $r_0 \le r_1 < r_2$, $0 < M_1 < M_2$, $\sigma > 1$, and $\theta > 1$ 
are some real numbers. Then
\begin{align}
	\left(
		\int_{M_1}^{M_2}
		(g_\theta (t) t)^{- 1 / p}
		\,
		dt
	\right)^{p / (p - 1)}
	&
	+
	\int_{M_1}^{M_2}
	g^{- 1 / \alpha}_\theta (t)
	\,
	dt
	\nonumber
	\\
	&
	\ge
	C
	\int_{r_1}^{r_2}
	(\xi f_\sigma (\xi))^{1 / (p - 1)}
	\,
	d\xi,
	\label{l3.5.1}
\end{align}
where the constant $C > 0$ depends only on 
$n$, $p$, $\alpha$, $\theta$, $\sigma$, $C_1$, and $C_2$.
\end{lemma}

\begin{proof}
We denote $\eta = \theta^{1 / 2}$ and $\tau = \sigma^{1 / 2}$.
Let $k$ be the maximal integer satisfying the condition $\tau^{k / 2} r_1 \le r_2$.
We put $\xi_i = \tau^{i / 2} r_1$, $i = 0, \ldots, k - 1$, and $\xi_k = r_2$.
It is obvious that
\begin{equation}
	\tau^{1 / 2} 
	\xi_i 
	\le 
	\xi_{i + 1} 
	\le 
	\tau 
	\xi_i,
	\quad
	i = 0, \ldots, k - 1.
	\label{pl3.5.7}
\end{equation}
In addition, for any $i \in \{ 0, \ldots, k - 1 \}$ at least one 
of the following three inequalities holds:
\begin{equation}
	\int_{
		M (\xi_i + 0; u)
	}^{
		M (\xi_{i + 1}; u)
	}
	(g_\eta (t) t)^{-1 / p}
	\,
	dt
	\ge
	c
	\int_{
		\xi_i
	}^{
		\xi_{i + 1}
	}
	q_\tau^{1 / p}
	(\xi)
	\,
	d\xi,
	\label{pl3.5.1}
\end{equation}
\begin{equation}
	\int_{
		M (\xi_i + 0; u)
	}^{
		M (\xi_{i + 1}; u)
	}
	g_\eta^{-1 / (p - 1)}
	(t)
	\,
	dt
	\ge
	c
	\int_{
		\xi_i
	}^{
		\xi_{i + 1}
	}
	(\xi f_\tau (\xi))^{1 / (p - 1)}
	\,
	d\xi,
	\label{pl3.5.2}
\end{equation}
\begin{equation}
	\int_{
		M (\xi_i + 0; u)
	}^{
		M (\xi_{i + 1}; u)
	}
	g_\eta^{-1 / \alpha}
	(t)
	\,
	dt
	\ge
	c
	\int_{
		\xi_i
	}^{
		\xi_{i + 1}
	}
	(\xi f_\tau (\xi))^{1 / (p - 1)}
	\,
	d\xi.
	\label{pl3.5.3}
\end{equation}
This follows from Lemma~\ref{l3.3} if 
$\eta^{1 / 2} M (\xi_i + 0; u) \le M (\xi_{i + 1}; u)$
or from Lemma~\ref{l3.4} otherwise.

By $\Xi_1$, $\Xi_2$, and $\Xi_3$ we denote the sets of integers $i \in \{ 0, \ldots, k - 1 \}$ 
satisfying relations~\eqref{pl3.5.1}, \eqref{pl3.5.2}, and~\eqref{pl3.5.3}, respectively.
Let
\begin{equation}
	\sum_{i \in \Xi_3}
	\int_{
		\xi_i
	}^{
		\xi_{i + 1}
	}
	(\xi f_\sigma (\xi))^{1 / (p - 1)}
	\,
	d\xi
	\ge
	\frac{1}{3}
	\int_{
		r_1
	}^{
		r_2
	}
	(\xi f_\sigma (\xi))^{1 / (p - 1)}
	\,
	d\xi.
	\label{pl3.5.4}
\end{equation}
Summing~\eqref{pl3.5.3} over all $i \in \Xi_3$, we obtain
$$
	\int_{
		M (r_1 + 0; u)
	}^{
		M (r_2; u)
	}
	g_\eta^{-1 / \alpha}
	(t)
	\,
	dt
	\ge
	c
	\sum_{
		i \in \Xi_3
	}
	\int_{
		\xi_i
	}^{
		\xi_{i + 1}
	}
	(\xi f_\tau (\xi))^{1 / (p - 1)}
	\,
	d\xi.
$$
The last inequality and~\eqref{pl3.5.4} obviously imply~\eqref{l3.5.1}
since $g_\tau (t) \ge g_\theta (t)$ for all $t > 0$ and, moreover,
$f_\tau (\xi) \ge f_\sigma (\xi)$ for all $\xi \in (r_1, r_2)$.
Assume that~\eqref{pl3.5.4} does not hold. 
In this case, we have either
\begin{equation}
	\sum_{i \in \Xi_1}
	\int_{
		\xi_i
	}^{
		\xi_{i + 1}
	}
	(\xi f_\sigma (\xi))^{1 / (p - 1)}
	\,
	d\xi
	\ge
	\frac{1}{3}
	\int_{
		r_1
	}^{
		r_2
	}
	(\xi f_\sigma (\xi))^{1 / (p - 1)}
	\,
	d\xi
	\label{pl3.5.5}
\end{equation}
or
\begin{equation}
	\sum_{i \in \Xi_2}
	\int_{
		\xi_i
	}^{
		\xi_{i + 1}
	}
	(\xi f_\sigma (\xi))^{1 / (p - 1)}
	\,
	d\xi
	\ge
	\frac{1}{3}
	\int_{
		r_1
	}^{
		r_2
	}
	(\xi f_\sigma (\xi))^{1 / (p - 1)}
	\,
	d\xi.
	\label{pl3.5.6}
\end{equation}
At first, let~\eqref{pl3.5.5} be valid. Taking into account~\eqref{pl3.5.7}, we obtain
\begin{align*}
	\left(
		\sum_{i \in \Xi_1}
		\int_{
			\xi_i
		}^{
			\xi_{i + 1}
		}
		q_\tau^{1 / p}
		(\xi)
		\,
		d\xi
	\right)^{p / (p - 1)}
	&
	\ge
	\sum_{i \in \Xi_1}
	\left(
		\int_{
			\xi_i
		}^{
			\xi_{i + 1}
		}
		q_\tau^{1 / p}
		(\xi)
		\,
		d\xi
	\right)^{p / (p - 1)}
	\\
	&
	\ge
	\sum_{i \in \Xi_1}
	(\xi_{i + 1} - \xi_i)^{p / (p - 1)}
	\inf_{
		(\xi_{i + 1}, \xi_i)
	}
	q_\tau^{1 / (p - 1)}
	\\
	&
	\ge
	c
	\sum_{i \in \Xi_1}
	\int_{
		\xi_i
	}^{
		\xi_{i + 1}
	}
	(\xi f_\sigma (\xi))^{1 / (p - 1)}
	\,
	d\xi.
\end{align*}
By~\eqref{pl3.5.1} and~\eqref{pl3.5.5}, this yields the estimate
$$
	\left(
		\int_{
			M (r_1 + 0; u)
		}^{
			M_(r_2; u)
		}
		(g_\eta (t) t)^{- 1 / p}
		\,
		dt
	\right)^{p / (p - 1)}
	\ge
	c
	\int_{r_1}^{r_2}
	(\xi f_\sigma (\xi))^{1 / (p - 1)}
	\,
	d\xi,
$$
whence~\eqref{l3.5.1} follows at once.

Now, let~\eqref{pl3.5.6} hold. Then, summing~\eqref{pl3.5.2} 
over all $i \in \Xi_2$, we conclude that
\begin{equation}
	\int_{
		M (r_1 + 0; u)
	}^{
		M (r_2; u)
	}
	g_\eta^{-1 / (p - 1)}
	(t)
	\,
	dt
	\ge
	c
	\int_{
		r_1
	}^{
		r_2
	}
	(\xi f_\sigma (\xi))^{1 / (p - 1)}
	\,
	d\xi.
	\label{pl3.5.8}
\end{equation}
Take the maximal integer $l$  satisfying the condition $\eta^{l / 2} M_1 \le M_2$.
We denote $t_i = \eta^{i / 2} M_1$, $i = 0, \ldots, l - 1$, and $t_l = M_2$. 
It can easily be seen that
$$
	\eta^{1 / 2} t_i \le t_{i + 1} \le \eta t_i,
	\quad
	i = 0, \ldots, l - 1.
$$
We have
\begin{align*}
	\left(
		\int_{
			M_1
		}^{
			M_2
		}
		(g_\theta (t) t)^{- 1 / p}
		\,
		dt
	\right)^{p / (p - 1)}
	&
	=
	\left(
		\sum_{i = 0}^{l - 1}
		\int_{
			t_i
		}^{
			t_{i + 1}
		}
		(g_\theta (t) t)^{- 1 / p}
		\,
		dt
	\right)^{p / (p - 1)}
	\\
	&
	\ge
	\sum_{i = 0}^{l - 1}
	\left(
		\int_{
			t_i
		}^{
			t_{i + 1}
		}
		(g_\theta (t) t)^{- 1 / p}
		\,
		dt
	\right)^{p / (p - 1)}
	\\
	&
	\ge
	\sum_{i = 0}^{l - 1}
	(t_{i + 1} - t_i)^{p / (p - 1)}
	t_{i + 1}^{- 1 / (p - 1)}
	\inf_{
		(t_i, t_{i + 1})
	}
	g_\theta^{- 1 / (p - 1)},
\end{align*}
whence in accordance with the inequality 
$$
	(t_{i + 1} - t_i)^{p / (p - 1)}
	t_{i + 1}^{- 1 / (p - 1)}
	\inf_{
		(t_i, t_{i + 1})
	}
	g_\theta^{- 1 / (p - 1)}
	\ge
	c
	\int_{
		t_i
	}^{
		t_{i + 1}
	}
	g_\eta^{- 1 / (p - 1)} (t)
	\,
	dt
$$
it follows that
\begin{align*}
	\left(
		\int_{
			M_1
		}^{
			M_2
		}
		(g_\theta (t) t)^{- 1 / p}
		\,
		dt
	\right)^{p / (p - 1)}
	&
	\ge
	c
	\sum_{i = 0}^{l - 1}
	\int_{
		t_i
	}^{
		t_{i + 1}
	}
	g_\eta^{- 1 / (p - 1)} (t)
	\,
	dt
	\\
	{}
	&
	=
	c
	\int_{
		M_1
	}^{
		M_2
	}
	g_\eta^{- 1 / (p - 1)} (t)
	\,
	dt.
\end{align*}
By~\eqref{pl3.5.8}, this implies the estimate
$$
	\left(
		\int_{
			M_1
		}^{
			M_2
		}
		(g_\theta (t) t)^{- 1 / p}
		\,
		dt
	\right)^{p / (p - 1)}
	\ge
	c
	\int_{
		r_1
	}^{
		r_2
	}
	(\xi f_\sigma (\xi))^{1 / (p - 1)}
	\,
	d\xi
$$
from which we immediately obtain~\eqref{l3.5.1}.

Lemma~\ref{l3.5} is completely proved.
\end{proof}

\begin{proof}[Proof of Theorem~$\ref{t2.1}$]
Assume to the contrary, that $u$ is a non-negative solution 
of~\eqref{1.1}, \eqref{1.2} and, moreover, $M (r_*; u) > 0$ for some $r_* > r_0$.
Lemma~\ref{l3.5} and condition~\eqref{t2.1.3} allows us to assert that 
$M (r; u) \to \infty$ as $r \to \infty$.
In so doing, the inequality
\begin{align}
	\left(
		\int_{
			M (r_*; u)
		}^{
			M (r; u)
		}
		(g_\theta (t) t)^{- 1 / p}
		\,
		dt
	\right)^{p / (p - 1)}
	&
	+
	\int_{
		M (r_*; u)
	}^{
		M (r; u)
	}
	g^{- 1 / \alpha}_\theta (t)
	\,
	dt
	\nonumber
	\\
	&
	\ge
	C
	\int_{r_*}^r
	(\xi f_\sigma (\xi))^{1 / (p - 1)}
	\,
	d\xi
	\label{pt2.1.1}
\end{align}
holds for all sufficiently large $r$, 
where the constant $C > 0$ depends only on 
$n$, $p$, $\alpha$, $\theta$, $\sigma$, $C_1$, and $C_2$.
Passing in this inequality to the limit as $r \to \infty$, 
we get a contradiction to conditions~\eqref{t2.1.1}--\eqref{t2.1.3}.

Theorem~\ref{t2.1} is completely proved.
\end{proof}

\begin{proof}[Proof of Theorem~$\ref{t2.2}$]
As in the proof of Theorem~\ref{t2.1}, we have
$M (r; u) \to \infty$ as $r \to \infty$. 
Hence, in formula~\eqref{pt2.1.1}, the real number $r_*$ can be taken such that 
$M (r_*; u) > 1$.
According to~\eqref{t2.1.3}, we also have
$$
	\int_{r_*}^r
	(\xi f_\sigma (\xi))^{1 / (p - 1)}
	\,
	d\xi
	\ge
	\frac{1}{2}
	\int_{r_0}^r
	(\xi f_\sigma (\xi))^{1 / (p - 1)}
	\,
	d\xi
$$
for all sufficiently large $r$.
Thus, estimate~\eqref{t2.2.1} follows at once from~\eqref{pt2.1.1}.

Theorem~\ref{t2.2} is completely proved.
\end{proof}

\begin{proof}[Proof of Theorem~$\ref{t2.3}$]
If $u \equiv 0$, then~\eqref{t2.3.1} is evident. 
Let $u \not\equiv 0$. In this case, it is obvious that $M (r; u) > 0$ 
for all $r$ in a neighborhood of infinity since ${M (\cdot; u)}$ is a non-decreasing function.
Consequently, applying Lemma~\ref{l3.5}, we obtain
\begin{align*}
	\left(
		\int_{
			M (r; u)
		}^\infty
		(g_\theta (t) t)^{- 1 / p}
		\,
		dt
	\right)^{p / (p - 1)}
	&
	+
	\int_{
		M (r; u)
	}^\infty
	g^{- 1 / \alpha}_\theta (t)
	\,
	dt
	\\
	&
	\ge
	C
	\int_r^\infty
	(\xi f_\sigma (\xi))^{1 / (p - 1)}
	\,
	d\xi
\end{align*}
for all sufficiently large $r$,
where the constant $C > 0$ depends only on 
$n$, $p$, $\alpha$, $\theta$, $\sigma$, $C_1$, and $C_2$.

To complete the proof, it remains to note that the last inequality is equivalent to~\eqref{t2.3.1}.
\end{proof}

\section{Proof of Lemmas~\ref{l3.0}--\ref{l3.2}}\label{Proof_of_Lemmas}

We extend the functions $A$ and $b$ in the left-hand side of~\eqref{1.1} by putting 
$
	{A (x, \xi)} 
	= 
	(C_1 + C_2)
	|\xi|^{p - 2}
	\xi
	/ 
	2
$
and
${b (x)} = 0$
for all 
$x \in {\mathbb R}^n \setminus \Omega$,
$\xi \in {\mathbb R}^n$.

Let us agree on the following notation.
In the proof of Lemma~\ref{l4.2}, by $c$ we denote various positive constants 
that can depend only on $n$, $p$, $\alpha$, and $\omega$.
In the proof of Lemmas~\ref{l4.4}, \ref{l4.5}, \ref{l4.6}, and \ref{l4.0}, 
analogous constants can depend only on $n$, $p$, $\alpha$, $C_1$, and $C_2$,
whereas in the proof of Lemmas~\ref{l4.7} and~\ref{l4.8}, they
can depend only on $n$, $p$, $\alpha$, $\gamma$, $C_1$, and $C_2$.
Finally, in the proof of Lemma~\ref{l4.10} these constants
can depend only on $n$, $p$, $\alpha$, $\gamma$, $\lambda$, $C_1$, and $C_2$.

Assume that $\omega_1$ and $\omega_2$ are open subsets of ${\mathbb R}^n$. 
A function
$v \in {W_{p, loc}^1 (\omega_1 \cap \omega_2)}$
satisfies the condition
$$
	\left.
	v
	\right|_{
		\omega_2 \cap \partial \omega_1
	}
	=
	0
$$
if
$
	\psi
	v
	\in
	{\stackrel{\rm \scriptscriptstyle o}{W}\!\!{}_p^1 (\omega_1 \cap \omega_2)}
$
for any $\psi \in C_0^\infty (\omega_2)$.
We also say that
\begin{equation}
	\left.
	v
	\right|_{
		\omega_2 \cap \partial \omega_1
	}
	\le
	0
	\label{4.1.1}
\end{equation}
if
$
	\psi
	\max \{ v, 0 \}    \in
	{\stackrel{\rm \scriptscriptstyle o}{W}\!\!{}_p^1 (\omega_1 \cap \omega_2)}
$
for any $\psi \in C_0^\infty (\omega_2)$.

\begin{lemma} \label{l4.1}
Let
$v \in W_{p, loc}^1 (\omega_1 \cap \omega_2)$
be a solution of the inequality
\begin{equation}
	\diver A (x, D v)
	+
	b (x) |D v|^\alpha
	\ge
	a (x)
	\quad
	\mbox{in }
	\omega_1 \cap \omega_2
	\label{l4.1.1}
\end{equation}
satisfying condition~\eqref{4.1.1},
where $a \in {L_{1, loc} (\omega_1 \cap \omega_2)}$.
We denote
$\tilde \omega = \{ x \in \omega_1 \cap \omega_2 : v (x) > 0 \}$,
$$
	\tilde v (x)
	=
	\left\{
		\begin{array}{ll}
		v (x),
		&
		x \in \tilde \omega,
		\\
		0,
		&
		x \in \omega_2 \setminus \tilde \omega
	\end{array}
	\right.
$$
and
$$
	\tilde a (x)
	=
	\left\{
		\begin{array}{ll}
		a (x),
		&
		x \in \tilde \omega,
		\\
		0,
		&
		x \in \omega_2 \setminus \tilde \omega.
	\end{array}
	\right.
$$
Then
$$
	\diver A (x, D \tilde v)
	+
	b (x) |D \tilde v|^\alpha
	\ge
	\tilde a (x)
	\quad
	\mbox{in } \omega_2.
$$
\end{lemma}

\begin{proof}
We take a non-decreasing function $\eta \in {C^\infty ({\mathbb R})}$ such that
$
	\left.
		\eta
	\right|_{
		(- \infty, 0]
	}
	=
	0
$
and
$
	\left.
		\eta
	\right|_{
		[1, \infty)
	}
	=
	1.
$
Put
$\eta_\tau (t) = \eta (t / \tau)$ 
and
$\varphi = \psi \eta_\tau (v)$,
where $\psi \in C_0^\infty (\omega_2)$ is a non-negative function
and $\tau > 0$ is a real number.
By~\eqref{l4.1.1}, we have
$$
	- \int_{
		\omega_1 \cap \omega_2
	}
	A (x, D v)
	D \varphi
	\, dx
	+
	\int_{
		\omega_1 \cap \omega_2
	}
	b (x) |D v|^\alpha
	\varphi
	\, dx
	\ge
	\int_{
		\omega_1 \cap \omega_2
	}
	a (x)
	\varphi
	\, dx,
$$
whence it follows that
\begin{align*}
	&
	- \int_{
		\omega_1 \cap \omega_2
	}
	\eta_\tau (u)
	A (x, D v)
	D \psi
	\, dx
	+
	\int_{
		\omega_1 \cap \omega_2
	}
	b (x) |D v|^\alpha
	\eta_\tau (u)
	\psi
	\, dx
	\\
	&
	\qquad
	{}
	\ge
	\int_{
		\omega_1 \cap \omega_2
	}
	a (x)
	\psi
	\eta_\tau (u)
	\, dx
	+
	\int_{
		\omega_1 \cap \omega_2
	}
	\eta_\tau' (u)
	\psi
	A (x, D v)
	D v
	\, dx
	\\
	&
	\qquad
	{}
	\ge
	\int_{
		\omega_1 \cap \omega_2
	}
	a (x)
	\psi
	\eta_\tau (u)
	\, dx.
\end{align*}
Passing to the limit as $\tau \to +0$ in the last expression, we obtain
\begin{align*}
	- \int_{
		\omega_1 \cap \omega_2
	}
	\chi_{
		\tilde \omega
	}
	(x)
	A (x, D v)
	D \psi
	\, dx
	&
	+
	\int_{
		\omega_1 \cap \omega_2
	}
	b (x) |D v|^\alpha
	\chi_{
		\tilde \omega
	}
	(x)
	\psi
	\, dx
	\\
	&
	\ge
	\int_{
		\omega_1 \cap \omega_2
	}
	\chi_{
		\tilde \omega
	}
	(x)
	a (x)
	\psi
	\, dx,
\end{align*}
where $\chi_{\tilde \omega}$ is the characteristic function of the set $\tilde \omega$, 
i.e. $\chi_{\tilde \omega} (x) = 1$ if $x \in \tilde \omega$ and 
$\chi_{\tilde \omega} (x) = 0$ otherwise.
Since
$$
	D \tilde v (x)
	=
	\left\{
		\begin{array}{ll}
		D v (x),
		&
		x \in \tilde \omega,
		\\
		0,
		&
		x \in \omega_2 \setminus \tilde \omega,
	\end{array}
	\right.
$$
this immediately implies that
$$
	- \int_{
		\omega_2
	}
	A (x, D \tilde v)
	D \psi
	\, dx
	+
	\int_{
		\omega_2
	}
	b (x) |D \tilde v|^\alpha
	\psi
	\, dx
	\ge
	\int_{
		\omega_2
	}
	\tilde a (x)
	\psi
	\, dx.
$$

The proof is completed.
\end{proof}

\begin{corollary}\label{c4.1}
Let $v \in W_p^1 (\omega)$ be a solution of the inequality
$$
	\diver A (x, D v)
	+
	b (x) |D v|^\alpha
	\ge
	a (x)
	\quad
	\mbox{in }
	\omega,
$$
where $\omega$ is an open subset of ${\mathbb R}^n$ and $a \in L_{1, loc} (\omega)$.
If
$\tilde \omega = \{ x \in \omega : v (x) > 0 \}$, 
$\tilde v = \chi_{\tilde \omega} v$
and, moreover,
$\tilde a = \chi_{\tilde \omega} a$,
then
$$
	\diver A (x, D \tilde v)
	+
	b (x) |D \tilde v|^\alpha
	\ge
	\tilde a (x)
	\quad
	\mbox{in }
	\omega.
$$
\end{corollary}

\begin{proof}
We put $\omega_1 = \omega_2 = \omega$ in Lemma~\ref{l4.1}.
\end{proof}

\begin{lemma}[the maximum principle]\label{l4.2}
Let
$v \in W_p^1 (\omega) \cap L_\infty (\omega)$
be a non-negative solution of the inequality
\begin{equation}
	\diver A (x, D v)
	+
	b (x) |D v|^\alpha
	\ge
	0
	\quad
	\mbox{in } \omega,
	\label{l4.2.1}
\end{equation}
where $\omega \subset {\mathbb R}^n$ is a bounded open set with a smooth boundary.
Then
\begin{equation}
	\esssup_\omega
	v
	=
	\esssup_{\partial \omega}
	v,
	\label{l4.2.2}
\end{equation}
where the restriction of $v$ to $\partial \omega$
is understood in the sense of the trace and
the $\esssup$ on the right-hand side is with respect to
$(n-1)$-dimensional Lebesgue measure on $\partial \omega$.
\end{lemma}

\begin{proof}
Assume that~\eqref{l4.2.2} is not valid. We put
$$
	v_\tau (x)
	=
	\max \{ v (x) - \tau, 0 \},
$$
where $\tau$ is a real number satisfying the condition
\begin{equation}
	\esssup_{\partial \omega}
	v
	<
	\tau
	<
	\esssup_\omega
	v.
	\label{pl4.2.1}
\end{equation}
It can easily be seen that $v_\tau$ is a non-negative function belonging to
$
	{\stackrel{\rm \scriptscriptstyle o}{W}\!\!{}_p^1 (\omega)};
$
therefore,
\begin{equation}
	- \int_\omega
	A (x, D v)
	D v_\tau
	\,	dx
	+
	\int_\omega
	b (x)
	| D v|^\alpha
	v_\tau
	\,	dx
	\ge
	0
	\label{pl4.2.2}
\end{equation}
in accordance with~\eqref{l4.2.1}.
Since 
$$
	D v_\tau (x)
	=
	\left\{
		\begin{array}{ll}
		D v (x),
		&
		x \in \omega_\tau,
		\\
		0,
		&
		x \in \omega \setminus \omega_\tau,
	\end{array}
	\right.
$$
where 
$\omega_\tau = \{ x \in \omega : \tau < v (x) < \esssup\nolimits_\omega v \}$, 
we have
$$
	\int_\omega
	A (x, D v)
	D v_\tau
	\,	dx
	=
	\int_{\omega_\tau}
	A (x, D v_\tau)
	D v_\tau
	\,	dx
	\ge
	C_1
	\int_{\omega_\tau}
	| D v_\tau |^p
	\,	dx
$$
and
$$
	\int_\omega
	b (x)
	| D v|^\alpha
	v_\tau
	\,	dx
	=
	\int_{\omega_\tau}
	b (x)
	| D v_\tau|^\alpha
	v_\tau
	\,	dx
	\le
	\esssup_\omega
	b
	\int_{\omega_\tau}
	| D v_\tau|^\alpha
	v_\tau
	\,	dx.
$$
Hence,~\eqref{pl4.2.2} allows us to assert that
\begin{equation}
	C_1
	\int_{\omega_\tau}
	| D v_\tau |^p
	\,	dx
	\le
	\esssup_\omega
	b
	\int_{\omega_\tau}
	| D v_\tau|^\alpha
	v_\tau
	\,	dx
	\label{pl4.2.3}
\end{equation}

At first, consider the case of $p - 1 \le \alpha < p$. 
Let $p_1 = p / \alpha$ and $p_2$ be some real number such that
$p < (p - \alpha) p_2 < n p / (n - p)$ if $n > p$ 
and 
$p < (p - \alpha) p_2$
if $n \le p$.
We also take the real number $p_3$ satisfying the relation
$1 / p_1 + 1 / p_2 + 1 / p_3 = 1$.
Since $1 / p_1 + 1 / p_2 < 1$, we obviously have $p_3 > 1$.

From the H\"older inequality for three functions, it follows that
\begin{align}
	\int_{\omega_\tau}
	| D v_\tau|^\alpha
	v_\tau
	\,	dx
	\le
	{}
	&
	\left(
		\int_{\omega_\tau}
		|D v_\tau|^p
		\,	dx
	\right)^{1 / p_1}
	\left(
		\int_{\omega_\tau}
		v_\tau^{(p - \alpha) p_2}
		\,	dx
	\right)^{1 / p_2}
	\nonumber
	\\
	&
	\times
	\left(
		\int_{\omega_\tau}
		v_\tau^{(1 - p + \alpha) p_3}
		\,	dx
	\right)^{1 / p_3}.
	\label{pl4.2.4}
\end{align}
Using the Friedrichs inequality and the Sobolev embedding theorem~\cite{Mazya}, 
we obtain
$$
	\left(
		\int_\omega
		v_\tau^{(p - \alpha) p_2}
		\,	dx
	\right)^{1 / p_2}
	\le
	c
	\left(
		\int_\omega
		|D v_\tau|^p
		\,	dx
	\right)^{(p - \alpha) / p}
$$
It is also obvious that
$$
	\int_\omega
	v_\tau^{(p - \alpha) p_2}
	\,	dx
	=
	\int_{\omega_\tau}
	v_\tau^{(p - \alpha) p_2}
	\,	dx
$$
and
$$
	\int_\omega
	|D v_\tau|^p
	\,	dx
	=
	\int_{\omega_\tau}
	|D v_\tau|^p
	\,	dx.
$$
Consequently,~\eqref{pl4.2.4} implies the estimate
$$
	\int_{\omega_\tau}
	| D v_\tau|^\alpha
	v_\tau
	\,	dx
	\le
	c
	\left(
		\int_{\omega_\tau}
		v_\tau^{(1 - p + \alpha) p_3}
		\,	dx
	\right)^{1 / p_3}
	\int_{\omega_\tau}
	|D v_\tau|^p
	\,	dx.
$$
Since 
$\mes \omega_\tau \to 0$ 
as 
$\tau \to \esssup\nolimits_\omega v - 0$, 
we have
$$
	\int_{\omega_\tau}
	v_\tau^{(1 - p + \alpha) p_3}
	\,	dx
	\le
	\mes \omega_\tau
	\esssup_{\omega_\tau}
	v_\tau^{(1 - p + \alpha) p_3}
	\to
	0
	\quad
	\mbox{as } \tau \to \esssup_\omega v - 0.
$$
Thus, by appropriate choice of the real number $\tau$,  one can achieve that
\begin{equation}
	\esssup_\omega
	b
	\int_{\omega_\tau}
	| D v_\tau|^\alpha
	v_\tau
	\,	dx
	\le
	\frac{C_1}{2}
	\int_{\omega_\tau}
	|D v_\tau|^p
	\,	dx.
	\label{pl4.2.5}
\end{equation}
Combining the last estimate with~\eqref{pl4.2.3}, we obtain
\begin{equation}
	\int_{\omega_\tau}
	| D v_\tau |^p
	\,	dx
	\le
	0,
	\label{pl4.2.6}
\end{equation}
whence it follows that $v_\tau = 0$
since 
$
	v_\tau
	\in
	{\stackrel{\rm \scriptscriptstyle o}{W}\!\!{}_p^1 (\omega)}.
$
This contradicts~\eqref{pl4.2.1}.

Now, let $\alpha = p$. Then
$$
	\int_{\omega_\tau}
	| D v_\tau|^\alpha
	v_\tau
	\,	dx
	\le
	\esssup_{\omega_\tau}
	v_\tau
	\int_{\omega_\tau}
	| D v_\tau|^p
	\,	dx.
$$
It is easy to see that
$$
	\esssup_{\omega_\tau}
	v_\tau
	=
	\esssup_\omega
	v
	-
	\tau
	\to
	0
	\quad
	\mbox{as } \tau \to \esssup_\omega v - 0;
$$
therefore, taking the real number $\tau$ close enough to $\esssup\nolimits_\omega v$,
we again obtain inequality~\eqref{pl4.2.5} which immediately implies~\eqref{pl4.2.6}.
Thus, we derive a contradiction once more.

The proof is completed.
\end{proof}

\begin{corollary}\label{c4.2}
Let $u \ge 0$ be a solution of problem~\eqref{1.1}, \eqref{1.2}, then
\begin{equation}
	M (r; u)
	=
	\esssup_{
		B_r
		\cap
		\Omega
	}
	u
	\label{c4.2.1}
\end{equation}
for all $r > r_0$.
\end{corollary}

\begin{proof}
We extend the function $u$ on the whole set ${\mathbb R}^n$ 
by putting $u = 0$ on ${\mathbb R}^n \setminus \Omega$.
According to Lemma~\ref{l4.1}, the extended function $\tilde u$ satisfies the inequality
$$
	\diver A (x, D \tilde u)
	+
	b (x) |D \tilde u|^\alpha
	\ge
	0
	\quad
	\mbox{in } {\mathbb R}^n.
$$
Thus, to obtain~\eqref{c4.2.1}, it remains to use Lemma~\ref{l4.2} with $\omega = B_r$.
\end{proof}

By $Q_l^z$ we mean the open cube in ${\mathbb R}^n$ 
of edge length $l > 0$ and center at a point~$z$.
In the case of $z = 0$, we write $Q_l$ instead of $Q_l^0$.

The following assertion is elementary, but useful.

\begin{lemma}\label{l4.3}
Let $v \in W_p^1 (Q_l^z)$, $l > 0$, $z \in {\mathbb R}^n$.
If $\mes \{ x \in Q_l^z : v (x) = 0 \} \ge l^n / 2$, then
$$	
	\int_{
		Q_l^z
	}
	v^p
	\, dx
	\le
	C
	l^p
	\int_{
		Q_l^z
	}
	|D v|^p
	\, dx,
$$
where the constant $C > 0$ depends only on $n$ and $p$.
\end{lemma}

\begin{proof}
Without loss of generality, it can be assumed that $l = 1$ and $z = 0$; 
otherwise we use the change of coordinates.

The proof is by reductio ad absurdum.
Let there be a sequence $v_k \in W_p^1 (Q_1)$
such that $\mes \{ x \in Q_1 : v_k (x) = 0 \} \ge 1 / 2$ and
\begin{equation}
	\int_{
		Q_1
	}
	v_k^p
	\, dx
	>
	k
	\int_{
		Q_1
	}
	|D v_k|^p
	\, dx,
	\quad
	k = 1, 2, \ldots.
	\label{pl4.3.1}
\end{equation}
It can be assumed that 
\begin{equation}
	\int_{
		Q_1
	}
	v_k^p
	\, dx
	=
	1,
	\quad
	k = 1, 2, \ldots;
	\label{pl4.3.2}
\end{equation}
otherwise we replace $v_k$ by $v_k / \| v_k \|_{L_p (Q_1)}$.
Therefore, $\{ v_k \}_{k = 0}^\infty$ is a bounded sequence in $W_p^1 (Q_1)$.
By the Sobolev embedding theorem~\cite{Mazya},
it has a subsequence that converges in $L_p (Q_1)$.
Taking into account~\eqref{pl4.3.1} and~\eqref{pl4.3.2}, 
one can obviously claim that this subsequence 
converges in the space $W_p^1 (Q_1)$ to a real number $\lambda \ne 0$.
To reduce clutter in indices, we also denote this subsequence by $\{ v_k \}_{k = 0}^\infty$.
Thus, we have
$$
	\| v_k - \lambda \|_{
		L_p (Q_1)
	}
	\to
	0
	\quad
	\mbox{as } k \to \infty.
$$
At the same time, from the definition of the functions $v_k$, it follows that
$$
	\| v_k - \lambda \|_{
		L_p (Q_1)
	}
	\ge
	|\lambda|
	\mes \{ x \in Q_1 : v_k (x) = 0 \} 
	\ge 
	\frac{|\lambda|}{2},
	\quad
	k = 1, 2, \ldots.
$$

This contradiction proves the lemma.
\end{proof}

\begin{lemma}\label{l4.4}
Let $v \in W_p^1 (B_r^z) \cap L_\infty (B_r^z)$, $r > 0$, $z \in {\mathbb R}^n$, 
be a non-negative solution of the inequality
$$
	\diver A (x, D v)
	+
	b (x) |D v|^\alpha
	\ge
	0
	\quad
	\mbox{in } B_r^z
$$
satisfying the condition
\begin{equation}
	r^{p - \alpha}
	\left(
		\esssup_{
			B_r^z
		}
		v
	\right)^{\alpha - p + 1}
	\esssup_{
		B_r^z
	}
	b
	\le
	\frac{C_1}{2}.
	\label{l4.4.2}
\end{equation}
Then
\begin{equation}
	\int_{
		B_{r_1}^z
	}
	|D v^\gamma|^p
	\, dx
	\le
	\frac{
		C
	}{
		(r_2 - r_1)^p
	}
	\int_{
		B_{r_2}^z
	}
	v^{\gamma p}
	\, dx
	\label{l4.4.3}
\end{equation}
for all real numbers $0 < r_1 < r_2 \le r$ and $\gamma \ge 1$,
where the constant $C > 0$ depends only on $n$, $p$, $\alpha$, $C_1$, and $C_2$.
\end{lemma}

\begin{proof}
We denote
\begin{equation}
	m
	=
	\esssup_{
		B_r^z
	}
	v,
	\quad
	\mu
	=
	\esssup_{
		B_r^z
	}
	b,
	\label{pl4.4.1}
\end{equation}
and $\psi (x) = \psi_0 ((|x - z| - r_1) / (r_2 - r_1))$, 
where $\psi_0 \in C^\infty ({\mathbb R})$ is a non-increasing function such that
$
	\left.
		\psi_0
	\right|_{
		(- \infty, 0]
	}
	=
	1
$
and
$
	\left.
		\psi_0
	\right|_{
		[1, \infty)
	}
	=
	0.
$
It can be assumed that $m > 0$; otherwise~\eqref{l4.4.3} is evident.
We also put $\beta = \gamma - (p - 1) / p$.
It is easy to see that $\beta \ge 1 / p$.
Taking $\varphi = \psi^p v^{\beta p}$ in the integral inequality
$$
	- \int_{
		B_r^z
	}
	A (x, D v)
	D \varphi
	\, dx
	+
	\int_{
		B_r^z
	}
	b (x)
	|D v|^\alpha
	\varphi
	\, dx
	\ge
	0,
$$
we obtain
\begin{align*}
	p
	\beta
	\int_{
		B_r^z
	}
	\psi^p
	v^{\beta p - 1}
	A (x, D v)
	D v
	\, dx
	\le
	{}
	&
	- p
	\int_{
		B_r^z
	}
	\psi^{p - 1}
	v^{\beta p}
	A (x, D v)
	D \psi
	\, dx
	\\
	&
	+
	\int_{
		B_r^z
	}
	b (x)
	|D v|^\alpha
	\psi^p
	v^{\beta p}
	\, dx.
\end{align*}
Since
$$
	C_1
	\int_{
		B_r^z
	}
	\psi^p
	v^{\beta p - 1}
	|D v|^p
	\, dx
	\le
	\int_{
		B_r^z
	}
	\psi^p
	v^{\beta p - 1}
	A (x, D v)
	D v
	\, dx
$$
and
$$
	\int_{
		B_r^z
	}
	\psi^{p - 1}
	v^{\beta p}
	|A (x, D v)|
	|D \psi|
	\, dx
	\le
	C_2
	\int_{
		B_r^z
	}
	\psi^{p - 1}
	v^{\beta p}
	|D v|^{p - 1}
	|D \psi|
	\, dx,
$$
this implies the estimate
\begin{align*}
	p
	\beta
	C_1
	\int_{
		B_r^z
	}
	\psi^p
	v^{\beta p - 1}
	|D v|^p
	\, dx
	\le
	{}
	&
	p
	C_2
	\int_{
		B_r^z
	}
	\psi^{p - 1}
	v^{\beta p}
	|D v|^{p - 1}
	|D \psi|
	\, dx
	\\
	&
	+
	\int_{
		B_r^z
	}
	b (x)
	|D v|^\alpha
	\psi^p
	v^{\beta p}
	\, dx
\end{align*}
from which, denoting 
$w (y) = (v (r y + z) / m)^\gamma$ 
and 
$\eta (y) = \psi  (r y + z)$,
we have
\begin{align}
	\frac{
		p
		\beta
		C_1
	}{
		\gamma^p
	}
	\int_{
		B_1
	}
	\eta^p
	|D w|^p
	\, dy
	\le
	{}
	&
	\frac{
		p
		C_2
	}{
		\gamma^{p - 1}
	}
	\int_{
		B_1
	}
	\eta^{p - 1}
	w
	|D w|^{p - 1}
	|D \eta|
	\, dy
	\nonumber
	\\
	&
	+
	\frac{
		r^{p - \alpha} 
		m^{\alpha - p + 1}
		\mu 
	}{
		\gamma^\alpha
	}
	\int_{
		B_1
	}
	|D w|^\alpha
	\eta^p
	w^{
		(\alpha - p + 1) / \gamma + p - \alpha
	}
	\, dy
	\nonumber
	\\
	\le
	{}
	&
	\frac{
		p
		C_2
	}{
		\gamma^{p - 1}
	}
	\int_{
		B_1
	}
	\eta^{p - 1}
	w
	|D w|^{p - 1}
	|D \eta|
	\, dy
	\nonumber
	\\
	&
	+
	\frac{
		C_1
	}{
		2 
		\gamma^\alpha
	}
	\int_{
		B_1
	}
	|D w|^\alpha
	\eta^p
	w^{
		(\alpha - p + 1) / \gamma + p - \alpha
	}
	\, dy
	\label{pl4.4.2}
\end{align}
in accordance with condition~\eqref{l4.4.2}.
From the Young inequality, it follows that
\begin{align*}
	\frac{
		p
		C_2
	}{
		\gamma^{p - 1}
	}
	\int_{
		B_1
	}
	\eta^{p - 1}
	w
	|D w|^{p - 1}
	|D \eta|
	\, dy
	&
	\le
	\frac{
		p
		\beta
		C_1
	}{
		4
		\gamma^p
	}
	\int_{
		B_1
	}
	\eta^p
	|D w|^p
	\, dy
	\\
	&
	\phantom{{}\le{}}
	+
	\frac{
		c
	}{
		\beta^{p - 1}
	}
	\int_{
		B_1
	}
	w^p
	|D \eta|^p
	\, dy.
\end{align*}
In the case of $\alpha < p$, the Young inequality also implies the relation
\begin{align*}
	\frac{
		C_1
	}{
		2 
		\gamma^\alpha
	}
	\int_{
		B_1
	}
	|D w|^\alpha
	\eta^p
	w^{
		(\alpha - p + 1) / \gamma + p - \alpha
	}
	\, dy
	\le
	{}
	&
	\frac{
		p
		\beta
		C_1
	}{
		2
		\gamma^p
	}
	\int_{
		B_1
	}
	\eta^p
	|D w|^p
	\, dy
	\\
	&
	+
	\frac{
		c
	}{
		\beta^{\alpha / (p - \alpha)}
	}
	\int_{
		B_1
	}
	\eta^p
	w^{(\alpha - p + 1) p / (\gamma (p - \alpha)) + p}
	\, dy
	\\
	\le
	{}
	&
	\frac{
		p
		\beta
		C_1
	}{
		2
		\gamma^p
	}
	\int_{
		B_1
	}
	\eta^p
	|D w|^p
	\, dy
	\\
	&
	+
	\frac{
		c
	}{
		\beta^{\alpha / (p - \alpha)}
	}
	\int_{
		B_1
	}
	\eta^p
	w^p
	\, dy.
\end{align*}
In turn, for $\alpha = p$, one can assert that
\begin{align*}
	\frac{
		C_1
	}{
		2 
		\gamma^\alpha
	}
	\int_{
		B_1
	}
	|D w|^\alpha
	\eta^p
	w^{
		(\alpha - p + 1) / \gamma + p - \alpha
	}
	\, dy
	&
	=
	\frac{
		C_1
	}{
		2 
		\gamma^\alpha
	}
	\int_{
		B_1
	}
	\eta^p
	|D w|^p
	w^{1 / \gamma}
	\, dx
	\\
	&
	\le
	\frac{
		C_1
	}{
		2
		\gamma^p
	}
	\int_{
		B_1
	}
	\eta^p
	|D w|^p
	\, dx.
\end{align*}
Therefore, taking into account~\eqref{pl4.4.2} and the fact that 
$\beta = \gamma - (p - 1) / p \ge 1 / p$, we obtain
$$
	\int_{
		B_r
	}
	\eta^p
	|D w|^p
	\, dy
	\le
	c
	\int_{
		B_r
	}
	(\eta^p + |D \eta|^p)
	w^p
	\, dy.
$$
This implies the estimate
$$
	\int_{
		B_{r_1 / r}
	}
	|D w|^p
	\, dy
	\le
	c
	\| \psi_0 \|_{
		C^1 ({\mathbb R})
	}
	\left(
		\frac{
			r
		}{
			r_2 - r_1
		}
	\right)^p
	\int_{
		B_{r_2 / r}
	}
	w^p
	\, dy,
$$
whence~\eqref{l4.4.3} follows at once. 

The proof is completed.
\end{proof}

\begin{lemma}[Moser's inequality]\label{l4.5}
Under the hypotheses of Lemma~$\ref{l4.4}$, we have
\begin{equation}
	\esssup_{
		B_{r / 2}^z
	}
	v
	\le
	C
	r^{- n / p}
	\left(
		\int_{
			B_r^z
		}
		v^p
		\, dx
	\right)^{1 / p},
	\label{l4.5.1}
\end{equation}
where the constant $C > 0$ depends only on $n$, $p$, $\alpha$, $C_1$, and $C_2$.
\end{lemma}

\begin{proof}
We take a real number $\lambda$ satisfying the conditions
$1 < \lambda < n / (n - p)$ in the case of $n > p$ 
and 
$1 < \lambda$ in the case of $n \le p$.
Also denote $r_k = 1 / 2 + 1 / 2^{k + 1}$ and $p_k = \lambda^k p$, $k = 0, 1, 2, \ldots$.

We use Moser's iterative process. 
By Lemma~\ref{l4.4}, 
$$
	\| D v^{\lambda^k} \|_{
		L_p (
			B_{
				r_{k + 1} r
			}^z
		)
	}
	\le
	\frac{
		2^k
		c
	}{
		r
	}
	\| v^{\lambda^k} \|_{
		L_p (
			B_{
				r_k r
			}^z
		)
	}
$$
for any non-negative integer $k$, whence it follows that
$$
	\| D w^{\lambda^k} \|_{
		L_p (
			B_{
				r_{k + 1}
			}
		)
	}
	\le
	2^k
	c
	\| w^{\lambda^k} \|_{
		L_p (
			B_{
				r_k
			}
		)
	},
$$
where $w (y) = v (r y + z)$. This implies the inequality
$$
	\| w^{\lambda^k} \|_{
		W_p^1 (
			B_{
				r_{k + 1}
			}
		)
	}
	\le
	2^k
	c
	\| w^{\lambda^k} \|_{
		L_p (
			B_{
				r_k
			}
		)
	}.
$$
At the same time,
$$
	\| w^{\lambda^k} \|_{
		L_{\lambda p} (
			B_{
				r_{k + 1}
			}
		)
	}
	\le
	c
	\| w^{\lambda^k} \|_{
			W_p^1 (
			B_{
				r_{k + 1}
			}
		)
	}
$$
by the Sobolev embedding theorem~\cite{Mazya}.
Thus, we obtain
$$
	\| w^{\lambda^k} \|_{
		L_{\lambda p} (
			B_{
				r_{k + 1}
			}
		)
	}
	\le
	2^k
	c
	\| w^{\lambda^k} \|_{
		L_p (
			B_{
				r_k
			}
		)
	}
$$
or, in other words,
$$
	\| w \|_{
		L_{p_{k + 1}} (
			B_{
				r_{k + 1}
			}
		)
	}
	\le
	(2^k c)^{\lambda^{- k}}
	\| w \|_{
		L_{p_k} (
			B_{
				r_k
			}
		)
	},
	\quad
	k = 0, 1, 2, \ldots.
$$
The last formula allows us to assert that
$$
	\| w \|_{
		L_\infty (
			B_{
				1 / 2
			}
		)
	}
	\le
	c
	\| w \|_{
		L_{p} (
			B_1
		)
	},
$$
whence we immediately derive~\eqref{l4.5.1}.

The proof is completed.
\end{proof}

\begin{lemma}\label{l4.6}
Under the hypotheses of Lemma~$\ref{l4.4}$,
for any real number $\varepsilon > 0$ there is a real number $\delta > 0$ 
depending only on $n$, $p$, $\alpha$, $\varepsilon$, $C_1$, and $C_2$
such that the condition
$\mes \{ x \in B_r^z : v (x) > 0 \} \le \delta r^n$
implies the inequality
\begin{equation}
	\esssup_{
		B_{r / 8}^z
	}
	v
	\le
	\varepsilon
	\esssup_{
		B_r^z
	}
	v.
	\label{l4.6.1}
\end{equation}
\end{lemma}

\begin{proof}
We denote $\omega = \mes \{ x \in B_r^z : v (x) > 0 \}$.
In the case of $\mes \omega = 0$, inequality~\eqref{l4.6.1} is obvious;
therefore, it can be assumed that $\mes \omega > 0$.

By Lemma~\ref{l4.4},
\begin{equation}
	\int_{
		B_{r / 2}^z
	}
	|D v|^p
	\, dx
	\le
	c
	r^{- p}
	\int_{
		B_r^z
	}
	v^p
	\, dx.
	\label{pl4.6.1}
\end{equation}

Let 
$\mes \omega < r^n / 2^{n + 1}$
and, moreover,
$k$ be the maximal positive integer such that
$\mes \omega < r^n / (2^{n + 1} k^n)$.
We take a family of disjoint open cubes 
$Q_{r / (2 k)}^{z_i}$,
$i = 1, \ldots, k^n$,
satisfying the condition
$
	\overline{Q_{r / 2}^z}
	= 
	\cup_{i = 1}^{k^n}
	\overline{Q_{r / (2 k)}^{z_i}}.
$
From Lemma~\ref{l4.3}, it follows that
$$	
	\int_{
		Q_{r / (2 k)}^{z_i}
	}
	v^p
	\, dx
	\le
	c
	k^{- p} r^p
	\int_{
		Q_{r / (2 k)}^{z_i}
	}
	|D v|^p
	\, dx,
	\quad
	i = 1, \ldots, k^n;
$$
therefore,
$$	
	\int_{
		Q_{r / 2}^z
	}
	v^p
	\, dx
	\le
	c
	k^{- p}
	r^p
	\int_{
		Q_{r / 2}^z
	}
	|D v|^p
	\, dx.
$$
Since $B_{r / 4}^z \subset Q_{r / 2}^z \subset B_{r / 2}^z$, 
the last inequality allows us to assert that
$$	
	\int_{
		B_{r / 4}^z
	}
	v^p
	\, dx
	\le
	c
	k^{- p}
	r^p
	\int_{
		B_{r / 2}^z
	}
	|D v|^p
	\, dx.
$$
Combining this with~\eqref{pl4.6.1}, we obtain
$$
	\int_{
		B_{r / 4}^z
	}
	v^p
	\, dx
	\le
	c
	k^{- p}
	\int_{
		B_r^z
	}
	v^p
	\, dx.
$$
At the same time, from Lemma~\ref{l4.5}, it follows that
$$
	\esssup_{
		B_{r / 8}^z
	}
	v
	\le
	c
	r^{- n / p}
	\left(
		\int_{
			B_{r / 4}^x
		}
		v^p
		\, dx
	\right)^{1 / p}.
$$
Thus, we have
$$
	\esssup_{
		B_{r / 8}^z
	}
	v
	\le
	\frac{
		c
		r^{- n / p}
	}{
		k
	}
	\left(
		\int_{
			B_r^z
		}
		v^p
		\, dx
	\right)^{1 / p}
	\le
	\frac{
		c
	}{
		k
	}
	\esssup_{
		B_r^z
	}
	v.
$$

To complete the proof, it remains to note that $k \to \infty$ as $\mes \omega \to 0$.
\end{proof}

\begin{lemma}\label{l4.7}
Let $v \in W_p^1 (B_r^z) \cap L_\infty (B_r^z)$, $r > 0$, $z \in {\mathbb R}^n$, 
be a non-negative solution of the inequality
\begin{equation}
	\diver A (x, D v)
	+
	b (x) |D v|^\alpha
	\ge
	\chi_\omega (x)
	h
	\quad
	\mbox{in } B_r^z
	\label{l4.7.1}
\end{equation}
satisfying condition~\eqref{l4.4.2}, 
where $h \ge 0$ is a real number,
$\omega = \{ x \in B_r^z : v (x) > 0 \}$,
and $\chi_\omega$ is the characteristic function of $\omega$.
If $\mes \omega > 0$ and, moreover,
\begin{equation}
	\lim_{\xi \to + 0}
	\esssup_{
		B_\xi^z
	}
	v
	\ge
	\gamma
	\esssup_{
		B_r^z
	}
	v
	\label{l4.7.2}
\end{equation}
for some real number $\gamma > 0$, then
$$
	\esssup_{
		B_r^z
	}
	v
	\ge
	C
	r^{p / (p - 1)}
	h^{1 / (p - 1)},
$$
where the constant $C > 0$ depends only on $n$, $p$, $\alpha$, $\gamma$, $C_1$, and $C_2$.
\end{lemma}

\begin{proof}
We put $\varphi (x) = \varphi_0 (|x - z| / r)$, 
where $\varphi_0 \in C^\infty ({\mathbb R})$ is a non-increasing function such that
$
	\left.
		\varphi_0
	\right|_{
		(- \infty, 1 / 4]
	}
	=
	1
$
and
$
	\left.
		\varphi_0
	\right|_{
		[1 / 2, \infty)
	}
	=
	0.
$
Also let $m$ and $\mu$ be the real numbers defined by~\eqref{pl4.4.1}.
It is clear that $m > 0$ since $\mes \omega > 0$.

Taking into account~\eqref{l4.7.1}, we obtain
$$
	- \int_{B_r^z}
	A (x, D v)
	D \varphi
	\, dx
	+
	\int_{B_r^z}
	b (x)
	|D v|^\alpha
	\varphi
	\, d x
	\ge
	\int_{B_r^z}
	\chi_\omega (x)
	h
	\varphi
	\, d x,
$$
whence it follows that
$$
	\int_{B_r^z}
	|A (x, D v)|
	|D \varphi|
	\, dx
	+
	\int_{B_r^z}
	b (x)
	|D v|^\alpha
	\varphi
	\, d x
	\ge
	h
	\mes 
	\{ 
		\omega 
		\cap 
		B_{r / 4}^z 
	\}.
$$
On the other hand,
$$
	\mes 
	\{ 
		\omega 
		\cap 
		B_{r / 4}^z 
	\}
	\ge
	c
	r^n
$$
in accordance with Lemma~\ref{l4.6} and relation~\eqref{l4.7.2}.
Therefore, one can claim that
\begin{equation}
	\int_{B_r^z}
	|A (x, D v)|
	|D \varphi|
	\, dx
	+
	\int_{B_r^z}
	b (x)
	|D v|^\alpha
	\varphi
	\, d x
	\ge
	c
	r^n
	h.
	\label{pl4.7.1}
\end{equation}
Using the H\"older inequality, we have
\begin{align*}
	\int_{B_r^z}
	|A (x, D v)|
	|D \varphi|
	\, dx
	&
	\le
	C_2
	\int_{
		B_{r / 2}^z
	}
	|D v|^{p - 1}
	|D \varphi|
	\, dx
	\\
	&
	\le
	C_2
	\left(
		\int_{
			B_{r / 2}^z
		}
		|D v|^p
		\, dx
	\right)^{(p - 1) / p}
	\left(
		\int_{
			B_{r / 2}^z
		}
		|D \varphi|^p
		\, dx
	\right)^{1 / p}
	\\
	&
	\le
	c
	r^{
		(n - p) / p
	}
	\left(
		\int_{
			B_{r / 2}^z
		}
		|D v|^p
		\, dx
	\right)^{(p - 1) / p}.
\end{align*}
At the same time,
$$
	\int_{B_r^z}
	b (x)
	|D v|^\alpha
	\varphi
	\, d x
	\le
	c
	r^{n (p - \alpha) / p}
	\mu
	\left(
		\int_{
			B_{r / 2}^z
		}
		|D v|^p
		\, d x
	\right)^{\alpha / p}.
$$
Really, the last estimate is obvious for $\alpha = p$, whereas in the case of $\alpha < p$, 
it follows from the H\"older inequality
$$
	\int_{B_r^z}
	b (x)
	|D v|^\alpha
	\varphi
	\, d x
	\le
	\mu
	\left(
		\int_{
			B_{r / 2}^z
		}
		|D v|^p
		\, d x
	\right)^{\alpha / p}
	\left(
		\int_{
			B_{r / 2}^z
		}
		\varphi^{p / (p - \alpha)}
		\, d x
	\right)^{(p - \alpha) / p}.
$$
Hence, formula~\eqref{pl4.7.1} allows us to assert that
$$
	r^{
		(n - p) / p
	}
	\left(
		\int_{
			B_{r / 2}^z
		}
		|D v|^p
		\, dx
	\right)^{(p - 1) / p}
	+
	r^{n (p - \alpha) / p}
	\mu
	\left(
		\int_{
			B_{r / 2}^z
		}
		|D v|^p
		\, d x
	\right)^{\alpha / p}
	\ge
	c
	r^n
	h.
$$
Since
$$
	\int_{
		B_{r / 2}^z
	}
	|D v|^p
	\, dx
	\le
	c
	r^{- p}
	\int_{
		B_r^z
	}
	v^p
	\, dx
	\le
	c
	r^{n - p}
	m^p
$$
by Lemma~\ref{l4.4}, this implies the estimate
$$
	m^{p - 1}
	+
	r^{p - \alpha}
	m^\alpha
	\mu
	\ge
	c
	r^p
	h,
$$
whence in accordance with the condition
$$
	r^{p - \alpha}
	m^\alpha
	\mu
	\le 
	\frac{
	C_1
	}{
		2
	}
	m^{p - 1}
$$
which follows from~\eqref{l4.4.2} we obtain
$$
	m^{p - 1}
	\ge
	c
	r^p
	h.
$$

The proof is completed.
\end{proof}

\begin{lemma}\label{l4.0}
Let $v \in W_p^1 (B_{r_1, r_2}) \cap L_\infty (B_{r_1, r_2})$
be a non-negative solution of the inequality
\begin{equation}
	\diver A (x, D v)
	+
	b (x) |D v|^\alpha
	\ge
	\chi_\omega (x)
	h
	\quad
	\mbox{in } B_{r_1, r_2},
	\label{l4.0.1}
\end{equation}
where  $0 < r_1 < r_2$ and $h \ge 0$  are real numbers,
$\omega = \{ x \in B_{r_1, r_2} : v (x) > 0 \}$,
and $\chi_\omega$ is the characteristic function of $\omega$.
If there exists a real number $m > 0$ such that ${M (r; v)} = m$ for all $r \in (r_1, r_2)$,
then $h = 0$.
\end{lemma}

\begin{proof}
We put $r_* = (r_1 + r_2) / 2$. 
There is a point $z \in S_{r_*}$ such that
$$
	\lim_{\xi \to + 0}
	\esssup_{
		B_\xi^z
	}
	v
	=
	m.
$$
Take real numbers $0 < \varepsilon < m$ and $0 < r < (r_2 - r_1) / 2$
satisfying the relation
$$
	\varepsilon^{\alpha - p + 1}
	r^{p - \alpha}
	\esssup_{
		B_{r_1, r_2}
	}
	b
	\le
	\frac{
		C_1
	}{
		4
	}.
$$
We denote
$$
	\tilde v (x)
	=
	\max
	\{
		v (x) - m + \varepsilon,
		0
	\}.
$$
By Corollary~\ref{l4.1}, $\tilde v$ is a non-negative solution of the inequality
\begin{equation}
	\diver A (x, D \tilde v)
	+
	b (x) |D \tilde v|^\alpha
	\ge
	\chi_{\tilde \omega} (x)
	h
	\quad
	\mbox{in } B_r^z,
	\label{pl4.0.1}
\end{equation}
where $\tilde \omega = \{x \in B_r^z : \tilde v (x) > 0 \} \subset \omega$.
Thus, 
$$
	\varepsilon
	\ge
	c
	r^{p / (p - 1)}
	h^{1 / (p - 1)}
$$
by Lemma~\ref{l4.7}.
Passing in the last expression to the limit as $\varepsilon \to +0$, we complete the proof.
\end{proof}

\begin{lemma}\label{l4.8}
Let $v \in W_p^1 (B_{r_1, r_2}) \cap L_\infty (B_{r_1, r_2})$
be a non-negative solution of~\eqref{l4.0.1}
such that ${M (\cdot; v)}$ is a non-decreasing function on $[r_1, r_2]$,
${M (r_2; v)} > {M (r_1; v)} \ge \gamma {M (r_2; v)}$
and, moreover,
\begin{equation}
	(M (r_2; v) - M (r_1; v))^{\alpha - p + 1}
	(r_2 - r_1)^{p - \alpha}
	\esssup_{
		B_{r_1, r_2}
	}
	b
	\le
	\frac{
		C_1
	}{
		4
	},
	\label{l4.8.2}
\end{equation}
where  $0 < r_1 < r_2$, $0 < \gamma < 1$, and $h \ge 0$  are real numbers,
$\omega = \{ x \in B_{r_1, r_2} : v (x) > 0 \}$,
and $\chi_\omega$ is the characteristic function of $\omega$.
Then
\begin{equation}
	(M (r_2; v) - M (r_1; v))^{p - 1}
	\ge
	C
	(r_2 - r_1)^p
	h,
	\label{l4.8.3}
\end{equation}
where the constant $C > 0$ depends only on $n$, $p$, $\alpha$, $\gamma$, $C_1$, and $C_2$.
\end{lemma}

\begin{proof}
We denote 
$r = (r_2 - r_1) / 2$ and $r_* = (r_1 + r_2) / 2$.
Take a point $z \in S_{r_*}$ such that
$$
	\lim_{\xi \to + 0}
	\esssup_{
		B_\xi^z
	}
	v
	\ge
	M (r_*; v).
$$
By the maximum principle, we have
$$
	M (r_2; v)
	=
	\esssup_{
		B_{r_1, r_2}
	}
	v
	\ge
	\esssup_{
		B_r^z
	}
	v.
$$
It can also be seen that
$$
	M (r_*; v)
	\ge
	M (r_1; v)
	\ge
	\gamma
	M (r_2; v).
$$
If ${M (r_2; v)} = {M (r_*; v)}$, then~\eqref{l4.8.3} is obvious since $h = 0$ by Lemma~\ref{l4.0}.
Hence, one can assume that ${M (r_2; v)} > {M (r_*; v)}$.
In the case of ${M (r_2; v)} \ge {2 M (r_*; v)}$, we obtain
$$
	\esssup_{
		B_r^z
	}
	v
	\le
	M (r_2; v)
	\le
	2 (M (r_2; v) - M (r_*; v));
$$
therefore,
\begin{align*}
	r^{p - \alpha}
	\left(
		\esssup_{
			B_r^z
		}
		v
	\right)^{\alpha - p + 1}
	\esssup_{
		B_r^z
	}
	b
	\le
	{}
	&
	2^{1 + 2 (\alpha - p)}
	(M (r_2; v) - M (r_*; v))^{\alpha - p + 1}
	\\
	&
	{\times{}}
	(r_2 - r_1)^{p - \alpha}
	\esssup_{
		B_{r_1, r_2}
	}
	b
	\\
	\le
	{}
	&
	\frac{C_1}{2}
\end{align*}
in accordance with~\eqref{l4.8.2}. Thus, Lemma~\ref{l4.7} allows us to assert that
$$
	M (r_2; v) - M (r_*; v)
	\ge
	\frac{1}{2}
	\esssup_{
		B_r^z
	}
	v
	\ge
	c
	r^{p / (p - 1)}
	h^{1 / (p - 1)},
$$
whence~\eqref{l4.8.3} follows at once.

Now, let ${M (r_2; v)} < {2 M (r_*; v)}$.
We put
$$
	\tilde v (x)
	=
	\max
	\{
		v (x) + {M (r_2; v)} - 2 {M (r_*; v)},
		0
	\}.
$$
By Corollary~\ref{c4.1}, the function $\tilde v$ is a non-negative solution 
of inequality~\eqref{pl4.0.1},
where $\tilde \omega = \{x \in B_r^z : \tilde v (x) > 0 \} \subset \omega$ as before.
In addition, we have
$$
	\lim_{\xi \to + 0}
	\esssup_{
		B_\xi^z
	}
	\tilde v
	\ge
	M (r_2; v) - M (r_*; v).
$$
and
$$
	\esssup_{
		B_r^z
	}
	\tilde v
	\le
	\esssup_{
		B_{r_1, r_2}
	}
	\tilde v
	=
	2 (M (r_2; v) - M (r_*; v)).
$$
Thus,~\eqref{l4.8.2} implies the estimate
\begin{align*}
	r^{p - \alpha}
	\left(
		\esssup_{
			B_r^z
		}
		{\tilde v}
	\right)^{\alpha - p + 1}
	\esssup_{
		B_r^z
	}
	b
	\le
	{}
	&
	2^{1 + 2 (\alpha - p)}
	(M (r_2; v) - M (r_*; v))^{\alpha - p + 1}
	\\
	&
	{\times{}}
	(r_2 - r_1)^{p - \alpha}
	\esssup_{
		B_{r_1, r_2}
	}
	b
	\\
	\le
	{}
	&
	\frac{C_1}{2}.
\end{align*}
To complete the proof, it remains to note that 
$$
	M (r_2; v) - M (r_*; v)
	\ge
	\frac{1}{2}
	\esssup_{
		B_r^z
	}
	\tilde v
	\ge
	c
	r^{p / (p - 1)}
	h^{1 / (p - 1)}
$$
according to Lemma~\ref{l4.7}.
\end{proof}

\begin{lemma}\label{l4.9}
Let $\eta : (r_1, r_2) \to [0, \infty)$ be a non-decreasing function such that
${\eta (r_1 + 0)} = 0$,
${\eta (r_2 - 0)} > \varepsilon$ and, moreover, 
${\eta (r - 0)} = {\eta (r)}$ for all $r \in (r_1, r_2)$,
where $r_1 < r_2$ and $\varepsilon > 0$ are real numbers. 
Then there is a real number $\xi \in (r_1, r_2)$ satisfying the conditions
${\eta (\xi)} \le \varepsilon$ and ${\eta (\xi + 0)} \ge \varepsilon$.
\end{lemma}

\begin{proof}
Consider sequences of real numbers $\{ \mu_i \}_{i=1}^\infty$ and $\{ m_i \}_{i=1}^\infty$
constructed by induction.
We put $\mu_1 = r_1$ and $m_1 = r_2$. Assume further that $\mu_i$ and $m_i$ are already known.
If $\eta ((\mu_i + m_i) / 2) > \varepsilon$, 
then we take $\mu_{i+1} = \mu_i$ and $m_{i+1} = (\mu_i + m_i) / 2$.
Otherwise we take $\mu_{i+1} =  (\mu_i + m_i) / 2$ and $m_{i+1} = m_i$.

It can easily be seen that 
$0 < m_i - \mu_i \le 2^{1 - i} (r_2 - r_1)$
and
$[\mu_{i+1}, m_{i+1}] \subset [\mu_i, m_i]$ 
for all $i = 1, 2, \ldots$.
Therefore, 
$$
	\lim_{i \to \infty} \mu_i
	=
	\lim_{i \to \infty} m_i
	=
	\xi
$$
for some $\xi \in [r_1, r_2]$.
If $\xi = r_1$, then 
$$
	\eta (r_1 + 0)
	=
	\lim_{i \to \infty} 
	\eta (m_i)
	\ge
	\varepsilon.
$$
This contradicts the fact that $\eta (r_1 + 0) = 0$.
On the other hand, if $\xi = r_2$, then 
$$
	\eta (r_2 - 0) 
	= 
	\lim_{i \to \infty} 
	\eta (\mu_i) 
	\le 
	\varepsilon
$$
and we derive a contradiction once more.
Thus, one can claim that $\xi \in (r_1, r_2)$.
In so doing, the relations
${\eta (\xi)} \le \varepsilon$ and ${\eta (\xi + 0)} \ge \varepsilon$
follow directly from the definition of the sequences 
$\{ \mu_i \}_{i=1}^\infty$ and $\{ m_i \}_{i=1}^\infty$.

The proof is completed.
\end{proof}

\begin{lemma}\label{l4.10}
In the conditions of Lemma~$\ref{l4.8}$, let the inequality
\begin{equation}
	(M (r_2; v) - M (r_1; v))^{\alpha - p + 1}
	(r_2 - r_1)^{p - \alpha}
	\esssup_{
		B_{r_1, r_2}
	}
	b
	\ge
	\lambda
	\label{l4.10.1}
\end{equation}
be valid instead of~\eqref{l4.8.2} for some real number $\lambda > 0$.
Then
\begin{equation}
	(M (r_2; v) - M (r_1; v))^\alpha
	\esssup_{
		B_{r_1, r_2}
	}
	b
	\ge
	C
	(r_2 - r_1)^\alpha
	h,
	\label{l4.10.2}
\end{equation}
where the constant $C > 0$ depends only on 
$n$, $p$, $\alpha$, $\gamma$, $\lambda$, $C_1$, and $C_2$.
\end{lemma}

\begin{proof}
It can be assumed that ${M (\cdot; v)}$ is a strictly increasing function on $[r_1, r_2]$;
otherwise $h = 0$ by Lemma~\ref{l4.0} and~\eqref{l4.10.2} is evident.
We denote 
$$
	\psi (r)
	=
	(M (r_2; v) - M (r; v))^{\alpha - p + 1}
	(r_2 - r)^{p - \alpha}
	\mu,
$$
where
$$
	\mu
	=
	\esssup_{
		B_{r_1, r_2}
	}
	b.
$$
Also let 
$\varepsilon = \min \{ \lambda, C_1 / 4 \}$ 
and
$r_* = \inf \{ r \in (r_1, r_2) : \psi (r) \le \varepsilon \}$.

By Lemma~\ref{l4.8}, we have
$$
	(M (r_2; v) - M (r_* + 0; v))^{p - 1}
	\ge
	c
	(r_2 - r_*)^p
	h.
$$
Multiplying this by the inequality
$$
	(M (r_2; v) - M (r_*; v))^{\alpha - p + 1}
	(r_2 - r_*)^{p - \alpha}
	\mu
	\ge
	\varepsilon
$$
which follows from~\eqref{l4.10.1} if $r_* = r_1$ or from the relation
${\psi (r_*)} = {\psi (r_* - 0)}$ if $r_* \in (r_1, r_2)$, we obtain
\begin{equation}
	(M (r_2; v) - M (r_*; v))
	\mu^{1 / \alpha}
	\ge
	c
	(r_2 - r_*)
	h^{1 / \alpha}.
	\label{pl4.10.2}
\end{equation}
In the case of $r_* = r_1$, we complete the proof.
Thus, it can be assumed that $r_* > r_1$. 
Le us construct a finite sequence of real numbers $\{ \xi_i \}_{i=1}^k$.
We put $\xi_1 = r_1$. Assume further that $\xi_i$ is already known. If
$$
	(M (r_2; v) - M (\xi_i + 0; v))^{\alpha - p + 1}
	(r_2 - \xi_i)^{p - \alpha}
	\mu
	\le
	\varepsilon,
$$
then we put $k = i$ and stop. Otherwise we take $\xi_{i+1} \in (\xi_i, r_2)$ such that
$$
	(M (\xi_{i+1}; v) - M (\xi_i + 0; v))^{\alpha - p + 1}
	(\xi_{i+1} - \xi_i)^{p - \alpha}
	\mu
	\le
	\varepsilon
$$
and
\begin{equation}
	(M (\xi_{i+1} + 0; v) - M (\xi_i + 0; v))^{\alpha - p + 1}
	(\xi_{i+1} - \xi_i)^{p - \alpha}
	\mu
	\ge
	\varepsilon.
	\label{pl4.10.4}
\end{equation}
By Lemma~\ref{l4.9}, such a real number $\xi_{i+1}$ obviously exists.
The above procedure must terminate at a finite step; 
otherwise, according to~\eqref{pl4.10.4}, we have
\begin{align*}
	M (\xi_{i+1} + 0; v) - M (\xi_i + 0; v)
	&
	\ge
	\left(
		\frac{
			\varepsilon
		}{	
			\mu
			(\xi_{i+1} - \xi_i)^{p - \alpha}
		}
	\right)^{1 / (\alpha - p + 1)}
	\\
	&
	\ge
	\left(
		\frac{
			\varepsilon
		}{	
			\mu
			(r_2 - r_1)^{p - \alpha}
		}
	\right)^{1 / (\alpha - p + 1)},
	\quad
	i = 1,2,\ldots,
\end{align*}
whence it follows that
$$
	M (r_2; v) - M (r_1; v)
	\ge
	\sum_{i=1}^\infty
	(M (\xi_{i+1} + 0; v) - M (\xi_i + 0; v))
	=
	\infty.
$$
By Lemma~\ref{l4.8}, 
$$
	(M (\xi_{i+1}; v) - M (\xi_i + 0; v))^{p - 1}
	\ge
	c
	(\xi_{i+1} - \xi_i)^p
	h,
	\quad
	i = 1, \ldots, k - 1.
$$
Multiplying this by~\eqref{pl4.10.4}, we obtain
$$
	(M (\xi_{i+1} + 0; v) - M (\xi_i + 0; v))
	\mu^{1 / \alpha}
	\ge
	c
	(\xi_{i+1} - \xi_i)
	h^{1 / \alpha},
	\quad
	i = 1, \ldots, k - 1.
$$
Therefore,
$$
	(M (\xi_k + 0; v) - M (r_1 + 0; v))
	\mu^{1 / \alpha}
	\ge
	c
	(\xi_k - r_1)
	h^{1 / \alpha}.
$$
Since $\xi_k \ge r_*$, the last estimate and~\eqref{pl4.10.2} allow us to assert that
$$
	(M (r_2; v) - M (r_1 + 0; v))
	\mu^{1 / \alpha}
	\ge
	c
	(r_2 - r_1)
	h^{1 / \alpha}.
$$

The proof is completed.
\end{proof}

\begin{proof}[Proof of Lemmas~$\ref{l3.0}$--$\ref{l3.2}$]
Corollary~\ref{c4.2} implies that ${M (\cdot; u)}$ is a non-decreasing function 
on $[r_1, r_2]$. 
We put
$
	\omega
	=
	\{ 
		x \in \Omega_{r_1, r_2} 
		: 
		u (x) > \beta M (r_1; u)
	\},
$
$$
	v (x)
	=
	\left\{
		\begin{array}{ll}
			u (x) - \beta M (r_1; u),
			&
			x \in \omega,
			\\
			0,
			&
			x \in \Omega_{r_1, r_2} \setminus \omega
		\end{array}
	\right.
$$
and
$$
	h
	=
	\essinf_{
		\Omega_{r_1, r_2}
	}
	q
	\inf_{
		(\beta M (r_1; u), M (r_2; u))
	}
	g.
$$
In so doing, for Lemma~\ref{l3.0}, we take $\beta = 1 / 2$.
By Lemma~\ref{l4.1}, $v$ is a non-negative solution of inequality~\eqref{l4.0.1}.
Thus, to complete the proof, it remains to apply 
Lemmas~\ref{l4.0}, \ref{l4.8}, and~\ref{l4.10}, respectively.
\end{proof}

\end{document}